\documentclass{amsart}
\usepackage{amsmath,amssymb,amsthm,amsfonts}
\usepackage{hyperref}
\usepackage{appendix}
\usepackage{graphicx}
\usepackage{setspace}
\usepackage{color}

\newtheorem{lemma}{Lemma}[section]
\newtheorem{theorem}{Theorem}[section]
\newtheorem{definition}{Definition}[section]
\newtheorem{proposition}{Proposition}[section]
\newtheorem{remark}{Remark}[section]

\numberwithin{equation}{section}

\date{
June 1, 2021}

\begin{document}
\title[Hydrostatic approximation and primitive equations justification]{The primitive equations approximation of the anisotropic  horizontally viscous Navier-Stokes equations}
\thanks{$^*$Corresponding author}
\author[J. Li]{Jinkai Li}
\address[J. Li]{
School of Mathematical Sciences, South China Normal University, Guangzhou 510631, China}
\email{jklimath@m.scnu.edu.cn; jklimath@gmail.com}

\author{Edriss~S.~Titi}
\address[Edriss~S.~Titi]{
Department of Mathematics, Texas A\&M University, 3368 TAMU, College Station,
TX 77843-3368, USA. Department of Applied Mathematics and Theoretical Physics, University of Cambridge, Cambridge CB3 0WA, U.K.
ALSO, Department of Computer Science and Applied Mathematics, Weizmann Institute of Science, Rehovot 76100, Israel.}
\email{titi@math.tamu.edu; Edriss.Titi@damtp.cam.ac.uk}

\author[G. Yuan]{Guozhi Yuan$^*$}
\address[G. Yuan]{
School of Mathematical Sciences, South China Normal University, Guangzhou 510631, China}
\email{shenggaoxii@163.com}

\begin{abstract}
In this paper, we provide rigorous justification of the hydrostatic approximation and the derivation of primitive equations as the small aspect ratio limit of  the incompressible three-dimensional
Navier-Stokes equations  in the anisotropic horizontal viscosity regime.
Setting $\varepsilon >0$ to be the small aspect ratio of the vertical to the horizontal scales of the domain, we investigate the case when the horizontal and vertical viscosities in the
incompressible three-dimensional Navier-Stokes equations are of orders $O(1)$ and $O(\varepsilon^\alpha)$, respectively, with
$\alpha>2$, for which the limiting system is the primitive equations with only horizontal viscosity as $\varepsilon$ tends to zero.
In particular we show that for ``well prepared" initial data the solutions of the scaled incompressible three-dimensional Navier-Stokes equations converge strongly, in
any finite interval of time, to the corresponding solutions of the anisotropic primitive equations with only horizontal viscosities, as $\varepsilon$
tends to zero, and that the convergence rate is of order $O\left(\varepsilon^\frac\beta2\right)$, where $\beta=\min\{\alpha-2,2\}$. Note that this result is different from the
case $\alpha=2$ studied in [Li, J.; Titi, E.S.: \emph{The primitive equations as the small aspect ratio limit of the
Navier-Stokes equations: Rigorous justification of the hydrostatic approximation}, J. Math. Pures Appl., \textbf{124}
\rm(2019), 30--58], where the limiting system is the primitive equations with full viscosities and the convergence is globally in time and its rate of order $O\left(\varepsilon\right)$.
\end{abstract}

\maketitle

\vspace{-5mm}

\allowdisplaybreaks

\section{Introduction}
The hydrostatic approximation is a fundamental assumption in the geophysics and a building block in the large scale
oceanic and atmospheric dynamics, see \cite{Lew,Maj,Ped,Val,Was,Zen}. It can be derived by either the scale analysis
or taking the small aspect ratio limit to the incompressible Navier-Stokes equations.
Thought it is proved to be accurate in the practical applications, the corresponding
rigorous mathematical justification has been only given in the case that
the horizontal and vertical viscosities have some particular orders of the aspect ratio,
see Az$\acute{e}$rad-Guill$\acute{e}$n   \cite{Aze} and Li-Titi \cite{Jli} in the weak and strong setting, respectively.
The aim of the current paper is to investigate the more general case that the horizontal and vertical viscosities are not
necessary to be of the particular order. As shown in the below
that the limiting system considered in the current paper is anisotropic primitive equations with only horizontal viscosities, while those in \cite{Aze,Jli} have full viscosities.

\subsection{Incompressible Navier-Stokes equations in thin domains}
Given a two dimensional domain $M=(0,L_1)\times(0,L_2)$ with $L_1,L_2>0$. Let
$\Omega_\varepsilon^-=M\times(-\varepsilon,0)$ be a three-dimensional box, where $\varepsilon >0$ is small representing the aspect ratio. Consider the anisotropic incompressible
Navier-Stokes equations in $\Omega_\varepsilon^-$
\begin{equation}
\left\{
\begin{aligned}
&\partial_t u+(u\cdot\nabla)u-\mu\Delta_Hu-\nu\partial_z^2u+\nabla p=0,\\
&\nabla\cdot u=0,
\end{aligned}
\right.\label{0}
\end{equation}
where the vector field $u=(v,w)$ representing the velocity, with $v=(v_1,v_2)$, and the scalar
function $p$ representing the pressure are the unknowns, $\mu$ and $\nu$ are the horizontal and vertical viscous
coefficients, respectively. Assume that $\mu=O(1)$ and $\nu=O(\varepsilon^\alpha)$ for some positive $\alpha$, as $\varepsilon \to 0$. The initial-boundary value problem will be studied in this paper and, thus, we complement system (\ref{0}) with the following boundary and initial conditions:
\begin{equation}\label{-}
\left\{
\begin{array}{l}
u~\textrm{and}~p~\textrm{are}~\textrm{periodic}~\textrm{in}~x~\textrm{and}~y,\\
(\partial_zv,w)|_{z=-\varepsilon,0}=(0,0),\\
u|_{t=0}=(v_0,w_0).
\end{array}
\right.
\end{equation}

Note that by extending $v$, $w$, and $p$, respectively, evenly, oddly, and evenly in $z$,
one can extend the initial-boundary value problem (\ref{0})--(\ref{-}) defined in
$\Omega_\varepsilon^-$ to the corresponding
problem defined in the extended domain $\Omega_\varepsilon:=M\times(-\varepsilon,\varepsilon)$.
The extended initial-boundary value problem in $\Omega_\varepsilon:=M\times(-\varepsilon,\varepsilon)$ is as follows
\begin{equation}
\left\{
\begin{aligned}
&\partial_t u+(u\cdot\nabla)u-\mu\Delta_Hu-\nu\partial_z^2u+\nabla p=0,\\
&\nabla\cdot u=0,\\
&v,~w~\textrm{and}~p~\textrm{are}~\textrm{periodic}~\textrm{in}~x,~y~\textrm{and}~z,\\
&v,~w~\textrm{and}~p~\textrm{be}~\textrm{even},~\textrm{odd}~\textrm{and}~\textrm{even}~\textrm{in}~z,\\
&u|_{t=0}=(v_0,w_0).
\end{aligned}
\right.\label{1}
\end{equation}
On the one hand, for any solution $(u,p)$
to (\ref{0})--(\ref{-}), if extending $v$, $w$, and $p$, respectively, evenly, oddly, and evenly in $z$, then the
extension, denoted by $(\tilde u, \tilde p)$, is a solution to (\ref{1}). In the setting of strong solutions, this
can be verified by noticing that extensions as above preserve the Sobolev
regularities of $v$ and $w$ due to the boundary conditions
in (\ref{-}), while in the setting of weak solutions, this is based on the fact that regular testing functions
satisfying the symmetry conditions in (\ref{1}) fulfill the boundary conditions in (\ref{-}) and thus can be chosen
as testing functions for (\ref{0})--(\ref{-}).
On the other hand, if $(u,p)$ is a solution to
(\ref{1}) in $\Omega_\varepsilon$, then the restriction of $(u,p)$ on $\Omega_\varepsilon^-$ is a solution to
(\ref{0})--(\ref{-}). Therefore (\ref{0})--(\ref{-}) is equivalent to (\ref{1}). Due to this equivalence,
one only needs to consider (\ref{1}).

We are interested in the small aspect ratio limit as $\varepsilon\rightarrow0$ to the above system.
Since only the regime of the primitive equations will be considered in the current paper, we assume that $\alpha\geq2$.
In fact, in the case $\alpha\in(0,2)$, one can show in a similar way as in \cite{Bre} that system \eqref{1} converges to a
limiting system with only vertical dissipation, which is different from the primitive equations.

In order to investigate the small aspect ratio limit, we first
carry out some scaling transformation to system \eqref{1} such that
the resulting system is defined on a fixed domain independent of $\varepsilon$. Similar to \cite{Jli}, we define the following new unknowns
\begin{equation*}
\begin{aligned}
&v_\varepsilon(x,y,z,t)=v(x,y,\varepsilon z,t),\qquad w_\varepsilon(x,y,z,t)=\frac{1}{\varepsilon}w(x,y,\varepsilon z,t),\\
&p_\varepsilon(x,y,z,t)=p(x,y,\varepsilon z,t),\qquad u_\varepsilon=(v_\varepsilon,w_\varepsilon),\quad \forall(x,y,z)\in M\times(-1,1).
\end{aligned}
\end{equation*}
Then, $u_\varepsilon$ and $p_\varepsilon$ satisfy the following scaled Navier-Stokes equations
\begin{equation}
\left\{
\begin{aligned}
&\partial_tv_\varepsilon+(u_\varepsilon\cdot\nabla)v_\varepsilon-\Delta_H v_\varepsilon-\varepsilon^{\alpha-2}\partial_z^2v_\varepsilon+\nabla_Hp_\varepsilon=0,\\
&\nabla\cdot u_\varepsilon=0,\\
&\varepsilon^2(\partial_tw_\varepsilon+u_\varepsilon\cdot\nabla w_\varepsilon-\Delta_Hw_\varepsilon-\varepsilon^{\alpha-2}\partial_z^2w_\varepsilon)+\partial_zp_\varepsilon=0,
\end{aligned}
\right.\label{2}
\end{equation}
in the fixed domain $\Omega:=M\times(-1,1)$, subject to
\begin{eqnarray}
&v_\varepsilon,~w_\varepsilon~\textrm{and}~p_\varepsilon~\textrm{are}~\textrm{periodic}~\textrm{in}~x,y,z,\label{3}\\
&v_\varepsilon,~w_\varepsilon~\textrm{and}~p_\varepsilon~\textrm{are}~\textrm{even},~\textrm{odd}~\textrm{and}~\textrm{even}
~\text{in}~z,~\textrm{respectively},\label{4}\\
&(v_\varepsilon,w_\varepsilon)|_{t=0}=(v_0,w_0).\label{5}
\end{eqnarray}
Since system (\ref{2}) preserves the above symmetry, one only needs to impose the required condition on the
initial velocity. Due to this, throughout this paper, we always assume that
\begin{equation}
v_0~\textrm{and}~w_0~\textrm{are}~\textrm{even}~\textrm{and}~\textrm{odd}~\textrm{in}~z,~\textrm{respectively}.\label{initial}
\end{equation}
Throughout this paper, we set $\nabla_H$ and $\Delta_H$ to denote $(\partial_x,\partial_y)$ and $\partial_x^2+\partial_y^2$, respectively. For any $1\leq q\leq \infty$ and positive integer $k$, we denote by $L^q(\Omega)$ and $H^k(\Omega)$, respectively, the standard Lebesgue and Sobolev spaces, and we use the notation $\|\cdot\|_q$ and $\|\cdot\|_{q,M}$ to denote the $L^q(\Omega)$ and $L^q(M)$ norms, respectively.
Since we consider the incompressible Navier-Stokes equations, we use $L^2_\sigma(\Omega)$ to denote the space consisting of all divergence-free functions in $L^2(\Omega)$. It should be emphasized that all the functions considered in this paper are supposed to be periodic in the spatial variables.

By the classic theory, see, e.g., \cite{Con} and \cite{Tem}, for any initial data $u_0\in L^2_\sigma(\Omega)$, there is a global weak solution $u$ to \eqref{2}, subject to \eqref{3} and \eqref{5}. Note that if the initial data $u_0$ satisfies
the symmetry condition (\ref{initial}), then one can construct, in the same way as in
\cite{Con} and \cite{Tem}, such weak solutions that satisfy the additional symmetry
condition (\ref{4}). In fact, in this case, the approximate solutions
satisfy the additional symmetry condition (\ref{4}) and, as a result, the weak solutions achieved as the limits of the approximated solutions also satisfy (\ref{4}). Therefore, for any $u_0\in L^2_\sigma(\Omega)$ satisfying the symmetry condition (\ref{initial}), there is global weak solution $u$ to system (\ref{2}) subject to (\ref{3})--(\ref{5}).
Here the weak solutions are defined as follows.

\begin{definition}
Let $u_0=(v_0,w_0)\in L^2_\sigma(\Omega)$ satisfy the symmetry condition (\ref{initial}). $u$ is called a Leray-Hopf weak solution to system \eqref{2} subject to
\eqref{3}--\eqref{5}, if

(i) $u\in C_w([0,\infty);L_\sigma^2(\Omega))\cap L_{loc}^2([0,\infty);H^1(\Omega))$ is spatially periodic and satisfies the
symmetry condition (\ref{4}), where $C_w$ means weakly continuity;

(ii) The following energy inequality holds:
\begin{equation*}
\begin{aligned}
\|v(t)\|_2^2+\varepsilon^2\|w(t)\|_2^2
&+2\int_0^t \Big(\|\nabla_Hv\|_2^2+\varepsilon^{\alpha-2}\|\partial_zv\|_2^2
+\varepsilon^2\|\nabla_Hw\|_2^2\\
&+\varepsilon^\alpha\|\partial_zw\|_2^2\Big)ds\leq \|v_0\|_2^2+\varepsilon^2\|w_0\|_2^2,\quad\mbox{for a.e.}\,t\in[0,\infty);
\end{aligned}
\end{equation*}

(iii) For any spatially periodic function $\varphi=(\varphi_H,\varphi_3)\in C_0^\infty(\overline{\Omega}\times[0,\infty))$ satisfying $\nabla\cdot \varphi=0$ and the
symmetry condition (\ref{4}), where $\varphi_H=(\varphi_1,\varphi_2)$, the following integral identity holds:
\begin{equation*}
\begin{aligned}
\int_0^\infty\int_\Omega&\Big[-(v\cdot\partial_t\varphi_H+\varepsilon^2w\partial_t\varphi_3)+(u\cdot\nabla)v\varphi_H+\varepsilon^2u\cdot\nabla w\varphi_3\\
&+\nabla_Hv:\nabla_H\varphi_H+\varepsilon^{\alpha-2}\partial_zv\cdot\partial_z\varphi_H+\varepsilon^2\nabla_Hw\cdot\nabla_H\varphi_3+\varepsilon^\alpha\partial_zw\partial_z\varphi_3\Big]d\Omega dt\\
=\int_\Omega& \Big(v_0\cdot\varphi_H(\cdot,0)+\varepsilon^2w_0\varphi_3(\cdot,0)\Big)d\Omega,
\end{aligned}
\end{equation*}
where $d\Omega=dxdydz$.
\end{definition}

\subsection{Small aspect ratio limit and the primitive equations (PEs)}
By taking the formal limit as $\varepsilon\rightarrow0$, it is natural to expect that
(\ref{2}) converges in some suitable sense to the following limiting systems
\begin{equation}
\left\{
\begin{aligned}
&\partial_tv+(u\cdot\nabla)v-\Delta v+\nabla_Hp=0,\\
&\nabla_H\cdot v+\partial_zw=0,\\
&\partial_zp=0,
\end{aligned}
\right.\label{6'}
\end{equation}
if $\alpha=2$ in (\ref{2}), and
\begin{equation}
\left\{
\begin{aligned}
&\partial_tv+(u\cdot\nabla)v-\Delta_Hv+\nabla_Hp=0,\\
&\nabla_H\cdot v+\partial_zw=0,\\
&\partial_zp=0,
\end{aligned}
\right.\label{6}
\end{equation}
if $\alpha>2$ in (\ref{2}),
where the vector field $u=(v,w)$ and the scalar function $p$ are the velocity and pressure, respectively.
Both (\ref{6'}) and (\ref{6}) are the simplest form of the primitive equations (PEs).
Note that in the case $\alpha=2$ the limiting system in (\ref{6'}) has
dissipation in all directions, while in the case $\alpha>2$ the corresponding
system in (\ref{6}) has dissipation only in the horizontal directions.

Recalling that we consider the periodic initial-boundary value problem to the scaled incompressible
Navier-Stokes equations \eqref{2}, it is clear that one should impose the same boundary conditions and symmetry
conditions to the corresponding limiting system (\ref{6}).
However, one only needs to impose the initial condition on the horizontal velocity.
In fact, by \eqref{initial}, $w_0$ is odd and periodic in $z$, one has $w_0|_{z=\pm1}=0$. Then, $w_0$ can
be uniquely determined by the incompressibility condition as
\begin{equation}
w_0(x,y,z)=-\int_{-1}^z\nabla_H\cdot v_0(x,y,\xi)d\xi,\quad \forall (x,y,z)\in\Omega.\label{9}
\end{equation}
We call initial data $(v_0,w_0)$ satisfying condition \eqref{9} {\it well prepared initial data}.

Similarly, $w$ can also be uniquely determined by the incompressibility condition as
\begin{equation}
w(x,y,z,t)=-\int_{-1}^z\nabla_H\cdot v(x,y,\xi,t)d\xi,\quad \forall (x,y,z)\in\Omega.\label{10}
\end{equation}
Due to these facts, throughout this paper, concerning the solutions to (\ref{6}), we only specify the horizontal components $v$, and $w$ is uniquely determined  by (\ref{10}).

The primitive equations, no matter with full or partial dissipation,
play fundamental roles in the geophysical fluid dynamics and, in particular, in the large scale
oceanic and atmospheric dynamics, one can see the books
\cite{Hal,Lew,Maj,Ped,Val,Was,Zen} for the applications and backgrounds of the primitive equations. They are the core in
the weather prediction models. Due to the presence of
strong turbulent mixing in the horizontal direction in the large scale atmosphere, the eddy viscosity in the horizontal
direction is much stronger than that in the vertical direction. As a result, both physically and
mathematically, it is necessary to investigate the primitive equations with anisotropic viscosities and, in particular,
the system that with only horizontal eddy viscosities.

The first systematical studies of the the primitive equation was made by Lions--Temam--Wang
\cite{Lio1,Lio2,Lio3} in the 1990s, where they established the global existence of weak solutions to the system that
with full viscosities; however, the uniqueness of weak solutions is still unclear, even for the two-dimensional case.
By making full use of the hydrostatic balance to exploit the two-dimensional structure of the key part of the pressure
and decomposing the velocity into barotropic and baroclinic components, Cao--Titi \cite{Cao7}
established the global well-posedness of strong solution to the three dimensional primitive equations, see also
Kobelkov \cite{Kob} and Kukavica--Ziane \cite{Kuk}. One can see \cite{Guo,Saa,Ju,Kuk1,Jli1} for the global well-posed results with weaker initial data, and see \cite{Lia} for the results taking the topography effects into considerations.  The global well-posedness results in
\cite{Cao7,Kob,Kuk} are established in the $L^2$ type spaces, for the corresponding results in the $L^p$ type spaces
based on the maximal regularity technique, one can see the works by Hieber et al. \cite{Hie1,Hie2} and Giga et al.
\cite{Gig1,Gig2}. Recently, global well-posedness of strong solutions to the coupled system of the primitive equations
to the moisture system with either one component or multi components of moisture, and the hydrostatic approximation from compressible Navier-Stokes equations to compressible primitive equations were also established, see \cite{Cot,Guo1,Hit1,Hit2} and \cite{Gao,Liu4}, respectively.
For the results of compressible primitive equations, one can see \cite{Liu1,Liu2,Liu3,KAZ,WFC1,WFC2}.

All the results mentioned in the above paragraph are for the system that with full dissipation.
In the last few years, some developments concerning the global well-posedness to the anisotropic primitive equations were
also made, see Cao--Titi \cite{Cao3} and Cao--Li--Titi \cite{Cao1,Cao2,Cao4,Cao5,Cao6}, which in particular imply that
the primitive equations with only horizontal viscosities are globally well-posed as long as one still has either
horizontal or vertical diffusivities, see also \cite{Han} and \cite{Saa}.
Notably, different from the primitive equations with either full viscosity or only
horizontal viscosity, the inviscid primitive equations may develop finite time singularities, see Cao et al. \cite{Tit1},
Wong \cite{Won}, Ghoul et al. \cite{Gho} and Ibrahim et al. \cite{Tit2}.

\subsection{Main results: rigourous justification of hydrostatic approximation}
As already mentioned at the beginning of this introduction, the rigorous justifications of the limiting process
in the case $\alpha=2$, i.e., the convergence from (\ref{2})
with $\alpha=2$ to (\ref{6'}) has been established by Az$\acute{e}$rad-Guill$\acute{e}$n \cite{Aze} in the
weak setting and by Li-Titi \cite{Jli} in the strong setting, respectively, see
also Furukawa et al. \cite{Fur1} and \cite{Fur2} for some generalizations in the $L^p$-$L^q$ type spaces.
To our best knowledge, the corresponding justification in the case $\alpha>2$, i.e., the convergence
from from (\ref{2}) with $\alpha>2$ to (\ref{6}), is still unknown, and we
are going to address this problem in the current paper.

Now, we are ready to state our main results.

We first consider the case that $v_0\in H^1(\Omega)$. In this case, noticing that $u_0$ can be only regarded as an
element in $L^2(\Omega)$ in general, one can only consider the
weak solutions to the anisotropic incompressible Navier-Stokes equations (\ref{2}).
For the primitive equations (\ref{6}), the local well-posed result in \cite{Cao1} guarantees a unique local in time
strong solution and, moreover, it can be extended to be a global one, if one has further that $\partial_zv_0\in L^m(\Omega)$ for some $m>2$.
As a result, we have the following local and global strong convergence results:

\begin{theorem}
\label{thm1}
Suppose that $\alpha>2$. Let $v_0\in H^1(\Omega)$ be a periodic function satisfying $\nabla_H\cdot\int_{-1}^1v_0dz=0$ on $M$. Assume that $v_0$ satisfies the symmetric condition \eqref{initial} and that $w_0$ is determined by \eqref{9}.
Denote by $(v_\varepsilon,w_\varepsilon)$ and $v$ an arbitrary Leray-Hopf weak solution to \eqref{2} and the unique
local strong solution to \eqref{6}, respectively, subject to \eqref{3}--\eqref{5} and with the same initial data $(v_0,w_0)$. Let $t_*$ be the time of existence of
$v$ and set
\begin{equation*}
(V_\varepsilon,W_\varepsilon,P_\varepsilon)=(v_\varepsilon-v,w_\varepsilon-w,p_\varepsilon-p).
\end{equation*}

Then, the following two items hold:

(i) It holds that
\begin{equation*}
\begin{aligned}
\sup_{0\leq t< t^*}\|V_\varepsilon,\varepsilon W_\varepsilon\|_2^2(t)
+\int_0^{t^*}\|\nabla_HV_\varepsilon,\varepsilon\nabla_HW_\varepsilon,\varepsilon^{\frac{\alpha-2}{2}}
\partial_zV_\varepsilon,\varepsilon^{\frac{\alpha}{2}}\partial_zW_\varepsilon\|_2^2\,dt
\leq& C\varepsilon^\beta,
\end{aligned}
\end{equation*}
for any $\varepsilon>0$ and $\alpha>2$, where $\beta:=\min\{2,\alpha-2\}$, and $C$ is a positive constant depending only on $\|v_0\|_{H^1}$, $t^*$, $L_1$ and $L_2$. As a direct consequence, one has
\begin{equation*}
\begin{aligned}
(v_\varepsilon,\varepsilon w_\varepsilon)&\rightarrow(v,0),~~\textrm{in}~L^\infty(0,t^*;L^2(\Omega)),\\
(\nabla_Hv_\varepsilon,\varepsilon^{\frac{\alpha-2}{2}}\partial_zv_\varepsilon,\varepsilon\nabla_Hw_\varepsilon,\varepsilon^{\frac{\alpha}{2}}&\partial_zw_\varepsilon,w_\varepsilon)\rightarrow(\nabla_Hv,0,0,0,w),~~\textrm{in}~L^2(0,t^*;L^2(\Omega)),
\end{aligned}
\end{equation*}
and the convergence rate is of the order $O(\varepsilon^{\frac{\beta}{2}})$.

(ii) Suppose in addition that $\partial_zv_0\in L^m(\Omega)$ for some $m>2$. Then, all the above
convergence and estimate still hold if replacing $t_*$ by any finite time $T\in(0,\infty)$. In particular,
it holds that
\begin{equation*}
\begin{aligned}
\sup_{0\leq t\leq T}\big(\|V_\varepsilon\|_2^2&+\varepsilon^2\|W_\varepsilon\|_2^2\big)(t)+\int_0^{T}
\Big(\|\nabla_HV_\varepsilon\|_2^2+\varepsilon^2\|\nabla_HW_\varepsilon\|_2^2\\
&+\varepsilon^{\alpha-2}\|\partial_zV_\varepsilon\|_2^2+\varepsilon^\alpha\|\partial_zW_\varepsilon\|_2^2\Big)dt\leq K(T)\varepsilon^\beta,
\end{aligned}
\end{equation*}
where $K$ is a nonnegative continuously increasing function on $[0,\infty)$ determined by $\|v_0\|_{H^1}$, $\|\partial_zv_0\|_m$, $L_1$, $L_2$, and $t^*$.
\end{theorem}

Next, we consider the case that $v_0\in H^2(\Omega)$. In this case, by \eqref{9}, it is clear that
$u_0=(v_0,w_0)\in H^1(\Omega)$. Then, by the local well-posedness theory of strong solutions to the incompressible
Navier-Stokes equations, see, e.g., \cite{Con,Tem}, for each $\varepsilon>0$, there is a unique local strong solution
$(v_\varepsilon,w_\varepsilon)$ to \eqref{2}, subject to \eqref{3}--\eqref{5}. For the primitive equations (\ref{6}), the
global well-posedness results in \cite{Cao1,Cao2} guarantee the global existence of strong solutions to \eqref{6},
subject to \eqref{3}--\eqref{5}. Then, we have the following strong convergence results.

\begin{theorem}
\label{thm2}
In addition to the conditions in Theorem \ref{thm1}, suppose that $v_0\in H^2(\Omega)$.
Denote by $(v_\varepsilon,w_\varepsilon)$ and $v$ the unique local strong solution to \eqref{2} and the unique global
strong solution to \eqref{6}, respectively, subject to \eqref{3}--\eqref{5} and with the same initial data $(v_0,w_0)$. Set
\begin{equation*}
(V_\varepsilon,W_\varepsilon)=(v_\varepsilon-v,w_\varepsilon-w),
\end{equation*}
and let $T_\varepsilon^*$ be the maximal time of existence of $(v_\varepsilon,w_\varepsilon)$.

Then, for any finite time $T>0$ and $\alpha>2$, there is a small positive constant $\varepsilon_T$ depending only on $\|v_0\|_{H^2}$,
$T$, $L_1$ and $L_2$, such that $T_\varepsilon^*>T$, as long as $\varepsilon\in(0,\varepsilon_T)$, and that
\begin{equation*}
\begin{aligned}
\sup_{0\leq t\leq T}\|V_\varepsilon,\varepsilon W_\varepsilon\|_{H^1}^2(t)
+\int_0^{T}\|\nabla_HV_\varepsilon,\varepsilon\nabla_HW_\varepsilon,\varepsilon^{\frac{\alpha-2}{2}}\partial_zV_\varepsilon,\varepsilon^{\frac{\alpha}{2}}\partial_zW_\varepsilon\|_{H^1}^2(t)dt
\leq& K_3(T)\varepsilon^\beta,
\end{aligned}
\end{equation*}
where $\beta=\min\{2,\alpha-2\}$ and $K_3$ is a nonnegative continuously increasing function on $[0,\infty)$
determined only by $\|v_0\|_{H^2}$, $L_1$ and $L_2$. As a consequence, one has
\begin{equation*}
\begin{aligned}
(v_\varepsilon,\varepsilon w_\varepsilon)&\rightarrow(v,0),~~\textrm{in}~L^\infty(0,T;H^1(\Omega)),\\
(\nabla_Hv_\varepsilon,\varepsilon^{\frac{\alpha-2}{2}}\partial_zv_\varepsilon,\varepsilon\nabla_Hw_\varepsilon,\varepsilon^{\frac{\alpha}{2}}&\partial_zw_\varepsilon,w_\varepsilon)\rightarrow(\nabla_Hv,0,0,0,w),~~\textrm{in}~L^2(0,T;H^1(\Omega)),\\
w_\varepsilon&\rightarrow w,~~\textrm{in}~L^\infty(0,T;L^2(\Omega)),
\end{aligned}
\end{equation*}
and the convergence rate is of the order $O(\varepsilon^{\frac{\beta}{2}})$.
\end{theorem}

\begin{remark}
Comparing with the results obtained in \cite{Jli}, where the strong convergence and error estimates are
globally in time or in other words uniformly in time for the primitive equations that with full dissipation,
the convergence and error estimates in the current
paper depend on the time intervals in which the problems are considered, as shown in Theorem \ref{thm1} and
Theorem \ref{thm2}. This is caused by the absence of the vertical
viscosity in the primitive equations (\ref{6}) which is treated carefully in the current paper, as both the strong convergence and
error estimates depend crucially on the a priori estimates for the relevant limiting system, i.e., the
primitive equations, while these a priori estimates available for the primite equations (\ref{6}) depend on
the time interval.
\end{remark}

It is interesting to compare the results in the case $\alpha>2$ with those in the case $\alpha=2$.
On the one hand, in the case $\alpha>2$, as shown in Theorem \ref{thm1} and Theorem \ref{thm2}, the
convergence rate $O(\varepsilon^{\frac{\beta}{2}})$, $\beta=\min\{2,\alpha-2\}$, becomes weaker and weaker
when $\alpha$ approaches $2$. On the other hand, in the case $\alpha=2$, the results in \cite{Jli} show
that the corresponding convergence rate is $O(\varepsilon)$. By comparing the results \cite{Jli} in the case $\alpha=2$
and our results, one may expect some better convergence rate, say $O(\varepsilon^{\kappa(\alpha)})$,
such that $\kappa(\alpha)\geq\kappa_0$ for some positive $\kappa_0$ when $\alpha$ approaches $2$.
Unfortunately, this seems impossible, as the following subtracted system for
$(V_\varepsilon, W_\varepsilon)$ has the quantity $\varepsilon^{\alpha-2}\partial_z^2v$
as a source term in the $V_\varepsilon$ equations:
\begin{equation*}
\begin{aligned}
\partial_tV_\varepsilon-\Delta_HV_\varepsilon-\varepsilon^{\alpha-2}\partial_z^2V_\varepsilon
+(U_\varepsilon\cdot\nabla)V_\varepsilon+\nabla_HP_\varepsilon&\\
+(U_\varepsilon\cdot\nabla)v+(u\cdot\nabla)V_\varepsilon=\varepsilon^{\alpha-2}\partial_z^2v&,\label{18}
\end{aligned}
\end{equation*}
\begin{equation*}
\nabla_H\cdot V_\varepsilon+\partial_zW_\varepsilon=0,\label{19}
\end{equation*}
\begin{equation*}
\begin{aligned}
\varepsilon^2(\partial_tW_\varepsilon-\Delta_HW_\varepsilon-\varepsilon^{\alpha-2}\partial_z^2W_\varepsilon
+U_\varepsilon\cdot\nabla W_\varepsilon+U_\varepsilon\cdot\nabla w+u\cdot\nabla W_\varepsilon)&\\
+\partial_zP_\varepsilon=-\varepsilon^2(\partial_tw-\Delta_Hw-\varepsilon^{\alpha-2}\partial_z^2w+u\cdot\nabla w)&. \label{20}
\end{aligned}
\end{equation*}
While in the case $\alpha=2$ as studied in \cite{Jli}, the corresponding subtracted system does not have any source terms
in $V_\varepsilon$ equations. These indicate the essential differences between the cases $\alpha>2$ and $\alpha=2$, or
in other words, the differences of the convergence from the incompressible Navier-Stokes equations to the
isotropic and anisotropic primitive equations.

The rest of this paper is arranged as follows: in section 2, we collect some preliminary results which will be used in the subsequent sections; in section 3, we cite some results about the local and global well-posedness of strong solutions to
the primitive equations with only horizontal viscosity and carry out some a priori estimates; finally, we give the proofs
of Theorem \ref{thm1} and Theorem \ref{thm2} in section 4 and section 5, respectively.

\section{Preliminaries}

The following inequality will be used frequently in the a priori estimates. Since it can be proved exactly in the same way as in \cite{Cao8} and \cite{Cao1}, we omit the proof here.
\begin{lemma}
\label{lem1}
The following trilinear inequalities hold:
\begin{equation*}
\begin{aligned}
&\ \ \  \int_M\left( \int_{-1}^1|\phi(x,y,z)|dz\right)\left(\int_{-1}^1|\varphi(x,y,z)\psi(x,y,z)|dz\right)dxdy\\
&\leq C\|\phi\|_{2}\|\varphi\|_{2}^{\frac{1}{2}}\Big(\|\varphi\|_{2}+\|\nabla_H\varphi\|_{2}\Big)^{\frac{1}{2}}
\|\psi\|_{2}^{\frac{1}{2}}\Big(\|\psi\|_{2}+\|\nabla_H\psi\|_{2}\Big)^{\frac{1}{2}}
\end{aligned}
\end{equation*}
and
\begin{equation*}
\begin{aligned}
&\ \ \  \int_M\left( \int_{-1}^1|\phi(x,y,z)|dz\right)\left(\int_{-1}^1|\varphi(x,y,z)\psi(x,y,z)|dz\right)dxdy\\
&\leq C\|\psi\|_{2}\|\varphi\|_{2}^{\frac{1}{2}}\Big(\|\varphi\|_{2}+\|\nabla_H\varphi\|_{2}\Big)^{\frac{1}{2}}
\|\phi\|_{2}^{\frac{1}{2}}\Big(\|\phi\|_{2}+\|\nabla_H\phi\|_{2}\Big)^{\frac{1}{2}}
\end{aligned}
\end{equation*}
here we still denote $\|\cdot\|_{q}=\|\cdot\|_{L^q(\Omega)}$, for any $\phi$, $\varphi$ and $\psi$, such that the quantities on the right hand sides are finite.
\end{lemma}

The following anisotropic Morrey inequality allows to control the H$\ddot{o}$lder norm by using different regularities in different directions.
\begin{lemma}
\label{lem2}
Let $\Omega=M\times(-1,1)$ and let $1\leq p_i<\infty~(i=1,2,3)$ with $\sum_{i=1}^3 p_i^{-1}<1$. Then, we have
\begin{equation*}
|\varphi|_{0,(\lambda_i)}\leq C\sum_{i=1}^3\|D_i\varphi\|_{p_i},~~~~\qquad\lambda_i=\frac{1-\sum_{j=1}^3p_j^{-1}}{1-\sum_{j=1}^3p_j^{-1}+3p_i^{-1}},
\end{equation*}
for any  $\varphi$  such that the quantities on the right hand sides are finite, where $C$ depends on $p_i$ and $\Omega$. Here $(D_1,D_2,D_3)=(\partial_x,\partial_y,\partial_z)$ and
\begin{equation*}
|\varphi|_{0,(\lambda_i)}=\sup_{x\in\bar{\Omega}}|\varphi(x)|+\sup_{x,y\in\bar{\Omega},x\neq y}\frac{|\varphi(x)-\varphi(y)|}{|x-y|^{(\lambda_i)}},
\end{equation*}
where
\begin{equation*}
|x-y|^{(\lambda_i)}=|x_1-y_1|^{\lambda_1}+|x_2-y_2|^{\lambda_2}+|x_3-y_3|^{\lambda_3}.
\end{equation*}
\end{lemma}
\begin{proof}
See \cite{Fai}.
\end{proof}

\section{Global well-posed of primitive equations with only horizontal viscosities}
The global well-posedness of strong solutions to the  primitive equations with only horizontal viscosities has been established in \cite{Cao1} and \cite{Cao2}. In this section, we improve slightly the result in \cite{Cao1}, see Proposition \ref{pro3}, below.

The following $H^1$ local well-posedness result is proved in \cite{Cao1}.

\begin{proposition}
\label{pro1}
Given a periodic function $v_0\in H^1(\Omega)$ with $\nabla_H\cdot\int_{-1}^1v_0 dz=0$ and satisfying the symmetric condition \eqref{initial}. Then,

(i) There is a unique local strong solution $v$ to \eqref{6}, subject to \eqref{3}--\eqref{5}.

(ii) The local existence time $t^*=\frac{6r_0^2\delta_0^2}{C_0}$, where $C_0$ depends only on $\delta_0$ and $r_0$,  $\delta_0 \in(0,1]$ and $r_0$ are  positive constants such that
\begin{equation*}
\sup_{x^H\in M}\int_{-1}^1\int_{D_{2r_0}(x^H)}|\partial_zv_0|^2~dxdydz\leq \delta_0^2.
\end{equation*}
Here we denote by $x^H$ a point in $\mathbb{R}^2$ and $D_{2r_0}(x^H)$  an open disk in $\mathbb{R}^2$ of radius $2r_0$ centered at $x^H$.

(iii) Moreover, the following estimate holds
\begin{equation*}
\sup_{0\leq t\leq t^*}\|v\|_{H^1}^2+\int_0^{t^*}\Big(\|\nabla_Hv\|_{H^1}^2+\|\partial_tv\|_2^2\Big)dt\leq C,
\end{equation*}
where the positive constant $C$ depends only on $t^*$, $\|v_0\|_{H^1}$, $L_1$ and $L_2$.
\end{proposition}
\begin{proof}
This is a direction consequence of Theorem 1.1 and Proposition 3.2 in \cite{Cao1}.
\end{proof}

Note that $\partial_z v$ has higher integrability in $[0,t^*]$, in case it has higher integrability at the initial time. In fact we have the following:
\begin{proposition}
\label{pro2}
Assume in addition to the conditions in Proposition \ref{pro1} that $\partial_zv_0\in L^m(\Omega)$ with $m>2$. Then, it holds that
\begin{equation*}
\sup_{0\leq t\leq t^*}\|\partial_zv\|_m\leq C\|\partial_zv_0\|_m,
\end{equation*}
where $C$ depends only on $m$, $t^*$, $\|v_0\|_{H^1}$, $L_1$ and $L_2$.
\end{proposition}
\begin{proof}
Set $v_z=\partial_zv$. Then, $v_z$ satisfies
\begin{equation*}
\partial_tv_z+v_z\cdot\nabla_Hv+v\cdot\nabla_Hv_z-\nabla_H\cdot v v_z-\left(\int_{-1}^z\nabla_H\cdot vd\xi\right)\partial_zv_z-\Delta_Hv_z=0.
\end{equation*}
Multiplying the above by $|v_z|^{m-2}v_z$, $m>2$, and integrating over $\Omega$, it follows from integrating by parts and the incompressibility condition that
\begin{equation*}
\begin{aligned}
\frac{1}{m}\frac{d}{dt}\int_\Omega|v_z|^m &d\Omega+\int_\Omega |v_z|^{m-2}\left(|\nabla_Hv_z|^2+(m-2)\Big|\nabla_H|v_z|\Big|^2\right)d\Omega\\
&=-\int_\Omega \Big(v_z\cdot\nabla_Hv|v_z|^{m-2}v_z-\nabla_H\cdot v v_z|v_z|^{m-2}v_z\Big)d\Omega\\
&\leq 2\int_\Omega |\nabla_Hv||v_z|^md\Omega\\
&\leq \int_M\left(2\int_{-1}^1|\nabla_Hv_z|dz+\int_{-1}^1|\nabla_Hv|dz\right)\left(\int_{-1}^1|v_z|^mdz\right)~dM:=I,
\end{aligned}
\end{equation*}
where the fact that $|\nabla_Hv|\leq \frac{1}{2}\int_{-1}^1|\nabla_Hv|dz+\int_{-1}^1|\nabla_Hv_z|dz$ has been used. It follows from Lemma \ref{lem1} and the Young inequality that
\begin{equation*}
\begin{aligned}
I
&\leq C\big(\|\nabla_Hv_z\|_2+\|\nabla_Hv\|_2\big)\big\||v_z|^{\frac{m}{2}}\big\|_2\big(\big\||v_z|^{\frac{m}{2}}\big\|_2+\big\|\nabla_H|v_z|^{\frac{m}{2}}\big\|_2\big)\\
&\leq\frac{1}{4}\int_\Omega |\nabla_Hv_z|^2|v_z|^{m-2}d\Omega +C\big(1+\|\nabla_Hv_z\|_2^2+\|\nabla_Hv\|_2^2\big)\|v_z\|_m^m.
\end{aligned}
\end{equation*}
As a result, it follows from Gronwall inequality that
\begin{equation*}
\sup_{0\leq t\leq t^*}\|v_z\|_m^m\leq e^{C\int_0^{t^*}(\|\nabla_Hv_z\|_2^2+\|\nabla_Hv\|_2^2+1)dt}\|\partial_zv_0\|_m^m,
\end{equation*}
which leads to the conclusion by Proposition \ref{pro1}.
\end{proof}

Now, we can extend the local strong solution to be a global one as stated in the following proposition. Note that in comparison to the global well-posedness result in \cite{Cao1}, the required condition $v_0\in L^\infty(\Omega)$ in \cite{Cao1} is removed here.
\begin{proposition}
\label{pro3}
Under the assumption of Proposition \ref{pro2}, the unique local strong solution $v$ stated in Proposition \ref{pro1} can be extended uniquely to be a global one such that for any finite time $T\in (0,\infty)$,
\begin{equation*}
\sup_{0\leq t\leq T}\|v\|_{H^1}^2+\int_0^{T}\Big(\|\nabla_Hv\|_{H^1}^2+\|\partial_tv\|_2^2\Big) dt\leq J(T),
\end{equation*}
where $J:[0,\infty)\mapsto\mathbb{R}^+$ is a continuously increasing function determined only by $\|v_0\|_{H^1}$, $\|\partial_zv_0\|_m$, $m$, $t^*$, $L_1$ and $L_2$. Here $t^*$ is given in Proposition \ref{pro1}.
\end{proposition}

\begin{proof}
Due to (iii) of Proposition \ref{pro1}, it has
\begin{equation*}
\int_{\frac{t^*}{2}}^{t^*}\Big(\|\nabla_H^2v\|_2^2+\|\nabla_H\partial_zv\|_2^2\Big)dt\leq C.
\end{equation*}
Choose a time $t'\in (\frac{t^*}{2},t^*)$ such that
\begin{equation*}
\|\nabla_H^2v\|_2^2(t')+\|\nabla_H\partial_zv\|_2^2(t')\leq \frac{C}{t^*}.
\end{equation*}
By the Sobolev imbedding inequality, this implies  $\|\nabla_Hv\|_6(t')\leq \frac{C}{t^*}$. Thanks to this and applying Lemma \ref{lem2} with $p_1=p_2=6$ and $p_3=2$, one obtains
\begin{equation*}
\sup_{x\in\bar{\Omega}}|v(x,t')|\leq C(2\|\nabla_Hv\|_6(t')+\|\partial_zv\|_2(t'))\leq\frac{C}{t^*},
\end{equation*}
 and in particular $v(t')\in L^\infty(\Omega)$. With the aid  of this and by (iii) of Proposition \ref{pro1} and Proposition \ref{pro2}, one has $v|_{t=t'}\in L^\infty(\Omega)\cap H^1(\Omega)$ and $\partial_zv|_{t=t'}\in L^m(\Omega)$. As a result, by viewing $t'$ as the initial time, one can apply the result in \cite{Cao1} to extend the local solution $v$ uniquely to be a global one and the corresponding estimate as stated in Proposition \ref{pro3} holds. The proof is complete.

\end{proof}

Finally, for the $H^2$ initial data, the following global well-posedness and a priori estimate are cited from \cite{Cao2}.
\begin{proposition}
\label{pro4}
Given a periodic function $v_0\in H^2(\Omega)$ with $\nabla_H\cdot\int_{-1}^1v_0 dz=0$
and satisfying the symmetric condition \eqref{initial}. Then, there is a unique global strong solution $v$ to \eqref{6}, subject to \eqref{3}--\eqref{5} and the following estimate holds
\begin{equation*}
\sup_{0\leq t\leq T}\|v\|_{H^2}^2+\int_0^{T}\Big(\|\nabla_Hv\|_{H^2}^2+\|\partial_tv\|_{H^1}^2\Big) dt\leq G(T),
\end{equation*}
where $G(T)$ is a continuously increasing function determined only by $\|v_0\|_{H^2}$, $L_1$ and $L_2$.
\end{proposition}

\section{Proof of Theorem \ref{thm1}}

Since $v_0\in H^1(\Omega)$ and recalling (\ref{9}), the initial data $u_0=(v_0,w_0)$ can only be regarded as an element of $L^2_\sigma(\Omega)$.
Thus, one needs to consider the weak form of the scaled Navier-Stokes equations \eqref{2}. By Proposition \ref{pro1}, the
unique solution $v$ to (\ref{6}), subject to (\ref{3})--(\ref{5}) has the regularities $v\in L^\infty(0,t^*;H^1(\Omega))$,
$\partial_tv\in L^2(0,t^*;L^2(\Omega))$, and $\nabla_Hv\in L^2(0,t^*;H^1(\Omega))$. Thanks to these facts, by virtue of a density
argument, one can check that $(v,w)$ can be chosen as testing function in the weak form in (iii) of Definition 1.1. As a result, we have the following proposition.

\begin{proposition}
\label{pro4.1}
Given a periodic function $v_0\in H^1(\Omega)$ with $\nabla_H\cdot\int_{-1}^1v_0 dz=0$ and satisfying the symmetric condition \eqref{initial}. Let $(v_\varepsilon,w_\varepsilon)$  an arbitrary Leray-Hopf weak solution to \eqref{2} and $v$  the unique local strong solution to \eqref{6} ,  subject to \eqref{3}--\eqref{5}. Then, the following integral equality holds
\begin{equation*}
\begin{aligned}
&-\frac{\varepsilon^2}{2}\|w(t_0)\|_2^2+\left[\int_\Omega\Big( v_\varepsilon\cdot v+\varepsilon^2 w_\varepsilon w\Big)~d\Omega\right](t_0)-\int_0^{t_0}\int_\Omega v_\varepsilon \partial_t v~d\Omega dt\\
&+\int_0^{t_0}\int_\Omega\Big(\nabla_Hv_\varepsilon:\nabla_Hv+\varepsilon^{\alpha-2}\partial_zv_\varepsilon\cdot \partial_zv+\varepsilon^2\nabla_Hw_\varepsilon\cdot\nabla_Hw+\varepsilon^\alpha\partial_zw_\varepsilon\partial_zw\Big)d\Omega dt\\
=&\|v_0\|_2^2+\frac{\varepsilon^2}{2}\|w_0\|_2^2+\varepsilon^2\int_0^{t_0}\int_\Omega\nabla_HW_\varepsilon\cdot\left(
\int_{-1}^z\partial_tvd\xi\right) d\Omega dt\\
&-\int_0^{t_0}\int_\Omega\Big((u_\varepsilon\cdot\nabla)v_\varepsilon v+\varepsilon^2u_\varepsilon\cdot w_\varepsilon w\Big)d\Omega dt,
\end{aligned}
\end{equation*}
for any $t_0\in[0,t^*]$, where $t^*$ is the time of existence of $v$.
\end{proposition}

\begin{proof}
The proof is exactly the same as in Proposition 4.1 of \cite{Jli} and, thus, it is omitted here.
\end{proof}

\begin{remark}
If we further assume that $\partial_zv_0\in L^m(\Omega)$, $m>2$, then by Proposition \ref{pro3}, for any finite time $T>0$, we can obtain the unique strong solution $v$ in $[0,T]$ to  \eqref{6}, and the result in Proposition \ref{pro4.1} holds for any finite time, in other words, one can replace $t^*$ by any positive time $T\in[0,\infty)$.
\end{remark}

Thanks to the Proposition \ref{pro4.1} and Remark 4.1, we are ready to establish the proof of Theorem \ref{thm1}.

\begin{proof}[Proof of Theorem \ref{thm1}.]

(i) It suffices to prove
\begin{equation}
\begin{aligned}
\sup_{0\leq t\leq t^*}\big(\|V_\varepsilon\|_2^2+\varepsilon^2\|W_\varepsilon\|_2^2\big)(t)
+\int_0^{t^*}\Big(\|\nabla_HV_\varepsilon\|_2^2+\varepsilon^2\|\nabla_HW_\varepsilon\|_2^2&\\
+\varepsilon^{\alpha-2}\|\partial_zV_\varepsilon\|_2^2+\varepsilon^\alpha\|\partial_zW_\varepsilon\|_2^2\Big)dt
\leq C(\|v_0\|_{H^1},L_1,L_2,t^*)\varepsilon^\beta&,\label{*}
\end{aligned}
\end{equation}
where $\beta=\min\{\alpha-2,2\}$.

As $v$ is the unique local strong solution of \eqref{6}, then \eqref{6} holds in $L^2(\Omega\times(0,t*))$ and consequently  one can multiply \eqref{6} by $v_\varepsilon$, and integrating over $\Omega\times (0,t_0)$. By integrating by parts, it follows
\begin{equation}
\int_0^{t_0}\int_\Omega\Big(\partial_tv\cdot v_\varepsilon+\nabla_Hv:\nabla_Hv_\varepsilon\Big)d\Omega dt=-\int_0^{t_0}\int_\Omega(u\cdot\nabla)v\cdot v_\varepsilon~d\Omega dt,\label{15}
\end{equation}
for any $t_0\in[0,t^*]$.
 Multiplying \eqref{6} by $v$ and integrating over $\Omega\times (0,t_0)$, it follows from integrating by parts that
\begin{equation}
\frac{1}{2}\|v(t_0)\|_2^2+\int_0^{t_0}\|\nabla_Hv\|_2^2dt=\frac{1}{2}\|v_0\|_2^2,\label{16}
\end{equation}
for any $t_0\in[0,t^*]$.
 The energy inequality in Definition 1.1 gives
\begin{equation}
\begin{aligned}
&\frac{1}{2}\big(\|v_\varepsilon(t_0)\|_2^2+\varepsilon^2\|w_\varepsilon(t_0)\|_2^2\big)\\
&\qquad+\int_0^{t_0}\Big(\|\nabla_Hv_\varepsilon\|_2^2+\varepsilon^{\alpha-2}\|\partial_zv_\varepsilon\|_2^2
+\varepsilon^2\|\nabla_Hw_\varepsilon\|_2^2+\varepsilon^\alpha\|\partial_zw_\varepsilon\|_2^2\Big)dt\\
\leq&\frac{1}{2}\big( \|v_0\|_2^2+\varepsilon^2\|w_0\|_2^2\big),\label{17}
\end{aligned}
\end{equation}
for a.e. $t_0\in[0,t^*]$, in particular for $t_0=0$. Summing \eqref{16} and \eqref{17}, and then subtracting \eqref{15} as well as the integral equality in Proposition \ref{pro4.1},
we obtain
\begin{equation*}
\begin{aligned}
&\frac{1}{2}\big(\|V_\varepsilon(t_0)\|_2^2+\varepsilon^2\|W_\varepsilon(t_0)\|_2^2\big)\\
&\qquad+\int_0^{t_0}\Big(\|\nabla_HV_\varepsilon\|_2^2+\varepsilon^2\|\nabla_HW_\varepsilon\|_2^2+\varepsilon^{\alpha-2}\|\partial_zV_\varepsilon\|_2^2+\varepsilon^\alpha\|\partial_zW_\varepsilon\|_2^2\Big)dt\\
\leq&-\int_0^{t_0}\int_\Omega\Big(\varepsilon^2\nabla_Hw\cdot\nabla_HW_\varepsilon+\varepsilon^{\alpha-2}\partial_zv\cdot\partial_zV_\varepsilon+\varepsilon^\alpha\partial_zw\partial_zW_\varepsilon\Big)d\Omega dt\\
&-\varepsilon^2\int_0^{t_0}\int_\Omega\nabla_HW_\varepsilon\cdot\left(\int_{-1}^z\partial_tv d\xi\right) d\Omega dt+\int_0^{t_0}\int_\Omega\varepsilon^2u_\varepsilon\cdot\nabla W_\varepsilon w~d\Omega dt\\
&+\int_0^{t_0}\int_\Omega\Big((u\cdot\nabla)v\cdot v_\varepsilon+(u_\varepsilon\cdot\nabla)v_\varepsilon\cdot v\Big)d\Omega dt:=I_1+I_2+I_3+I_4,
\end{aligned}
\end{equation*}
for a.e. $t_0\in[0,t^*]$.

$I_1$ and $I_2$ can be estimated directly by the H$\ddot{o}$lder and Young inequalities as
\begin{equation*}
\begin{aligned}
I_1&=-\int_0^{t_0}\int_\Omega\Big(\varepsilon^2\nabla_Hw\cdot\nabla_HW_\varepsilon+\varepsilon^{\alpha-2}\partial_zv\cdot\partial_zV_\varepsilon+\varepsilon^\alpha\partial_zw\partial_zW_\varepsilon\Big)d\Omega dt\\
&\leq \varepsilon^2\|\nabla_Hw\|_{L^2(Q_{t_0})}\|\nabla_HW_\varepsilon\|_{L^2(Q_{t_0})}+\varepsilon^{\alpha-2}\|\partial_zv\|_{L^2(Q_{t_0})}\|\partial_zV_\varepsilon\|_{L^2(Q_{t_0})}\\
&\qquad+\varepsilon^\alpha\|\partial_zw\|_{L^2(Q_{t_0})}\|\partial_zW_\varepsilon\|_{L^2(Q_{t_0})}\\
&\leq \frac{1}{8}\big(\varepsilon^2\|\nabla_HW_\varepsilon\|_{L^2(Q_{t_0})}^2+\varepsilon^{\alpha-2}\|\partial_zV_\varepsilon\|_{L^2(Q_{t_0})}^2+\varepsilon^\alpha\|\partial_zW_\varepsilon\|_{L^2(Q_{t_0})}^2\big)\\
&\qquad +C\varepsilon^\beta\big(\|\nabla_Hw\|_{L^2(Q_{t_0})}^2+\|\partial_zv\|_{L^2(Q_{t_0})}^2+\|\partial_zw\|_{L^2(Q_{t_0})}^2\big),
\end{aligned}
\end{equation*}
and
\begin{equation*}
\begin{aligned}
I_2&=-\varepsilon^2\int_0^{t_0}\int_\Omega\nabla_HW_\varepsilon\cdot\Big(\int_{-1}^z\partial_tv d\xi\Big) d\Omega dt\\
&\leq \frac{\varepsilon^2}{8}\|\nabla_HW_\varepsilon\|_{L^2(Q_{t_0})}^2+C\varepsilon^\beta\|\partial_tv\|_{L^2(Q_{t_0})}^2,
\end{aligned}
\end{equation*}
where $Q_{t_0}=\Omega\times(0,t_0)$.

By the incompressibility condition \eqref{10}, one obtains
\begin{equation*}
\begin{aligned}
I_3&=\varepsilon^2\int_0^{t_0}\int_\Omega u_\varepsilon\cdot\nabla W_\varepsilon w~d\Omega dt\\
&=\varepsilon^2\int_0^{t_0}\int_\Omega\Big( v_\varepsilon\cdot\nabla_HW_\varepsilon w-w_\varepsilon\big(\nabla_H\cdot V_\varepsilon\big)w\Big)d\Omega dt =:I_{31}+I_{32}.
\end{aligned}
\end{equation*}
For $I_{31}$ and $I_{32}$, by Lemma \ref{lem1} and using the Young inequality, one deduces
\begin{equation*}
\begin{aligned}
I_{31}\leq&\varepsilon^2\int_0^{t_0}\int_\Omega |v_\varepsilon||\nabla_HW_\varepsilon|\left(\int_{-1}^z|\nabla_H\cdot v|d\xi\right) d\Omega dt\\
\leq&\varepsilon^2\int_0^{t_0}\int_M\left(\int_{-1}^1 |v_\varepsilon||\nabla_HW_\varepsilon|dz\right)\left(\int_{-1}^1|\nabla_H v|dz\right) dM dt\\
\leq & C\varepsilon^2\int_0^{t_0}\|v_\varepsilon\|_2^{\frac{1}{2}}\big(\|v_\varepsilon\|_2+\|\nabla_Hv_\varepsilon\|_2\big)^{\frac{1}{2}}\|\nabla_HW_\varepsilon\|_2\|\nabla_Hv\|_2^{\frac{1}{2}}\|\Delta_Hv\|_2^{\frac{1}{2}}dt\\
\leq& C\varepsilon^2\int_0^{t_0}\Big[\|v_\varepsilon\|_2^2\big(\|v_\varepsilon\|_2^2+\|\nabla_Hv_\varepsilon\|_2^2\big)
+\|\nabla_Hv\|_2^{2}\|\Delta_Hv\|_2^{2}\Big]dt\\
&+\frac{\varepsilon^2}{8}\|\nabla_HW_\varepsilon\|_{L^2(Q_{t_0})}^2
\end{aligned}
\end{equation*}
and
\begin{equation*}
\begin{aligned}
I_{32}\leq&\varepsilon^2\int_0^{t_0}\int_\Omega |w_\varepsilon||\nabla_H V_\varepsilon|\left(\int_{-1}^z|\nabla_Hv|d\xi\right) d\Omega dt\\
\leq&\varepsilon^2\int_0^{t_0}\int_M\Big(\int_{-1}^1|w_\varepsilon||\nabla_H V_\varepsilon|dz\Big)\left(\int_{-1}^1|\nabla_Hv|dz\right) dM dt\\
\leq & C\varepsilon^2\int_0^{t_0}\|w_\varepsilon\|_2^{\frac{1}{2}}\|\nabla_Hw_\varepsilon\|_2^{\frac{1}{2}}\|\nabla_HV_\varepsilon\|_2\|\nabla_Hv\|_2^{\frac{1}{2}}\|\Delta_Hv\|_2^{\frac{1}{2}}dt\\
\leq&C\varepsilon^2\int_0^{t_0}\big(\varepsilon^4\|w_\varepsilon\|_2^2\|\nabla_Hw_\varepsilon\|_2^2+\|\nabla_Hv\|_2^2\|\Delta_Hv\|_2^2\big)dt\\
&+\frac{1}{8}\|\nabla_HV_\varepsilon\|_{L^2(Q_{t_0})}^2.
\end{aligned}
\end{equation*}
Therefore, combining the estimates of $I_{31}$ and $I_{32}$, one gets
\begin{equation*}
\begin{aligned}
I_3\leq \frac{\varepsilon^2}{8}&\|\nabla_HW_\varepsilon\|_{L^2(Q_{t_0})}^2+\frac{1}{8}\|\nabla_HV_\varepsilon\|_{L^2(Q_{t_0})}^2+C\varepsilon^2,
\end{aligned}
\end{equation*}
where we have used the result of Proposition \ref{pro1} and the energy inequality for $(v_\varepsilon,w_\varepsilon)$ in Definition 1.1.

Finally, for $I_4$, by the incompressibility condition and integrating by parts,  it follows
\begin{equation*}
\begin{aligned}
I_4&=\int_0^{t_0}\int_\Omega\Big(-(u\cdot\nabla)v_\varepsilon \cdot v+(u_\varepsilon\cdot\nabla)v_\varepsilon\cdot v\Big)d\Omega dt\\
&=\int_0^{t_0}\int_\Omega(U_\varepsilon\cdot\nabla)v_\varepsilon\cdot v~d\Omega dt
 =\int_0^{t_0}\int_\Omega(U_\varepsilon\cdot\nabla)V_\varepsilon\cdot v~d\Omega dt\\
&=\int_0^{t_0}\int_\Omega(V_\varepsilon\cdot\nabla_H)V_\varepsilon\cdot v~d\Omega dt+\int_0^{t_0}\int_\Omega W_\varepsilon\partial_zV_\varepsilon\cdot v~d\Omega dt =:I_{41}+I_{42}.
\end{aligned}
\end{equation*}
Using $|v|\leq \int_{-1}^1|\partial_zv|dz+\frac{1}{2}\int_{-1}^1|v|dz$, it follows from Lemma \ref{lem1} that
\begin{equation*}
\begin{aligned}
I_{41}\leq& \int_0^{t_0}\int_M\left(\int_{-1}^1|V_\varepsilon||\nabla_HV_\varepsilon|dz\right)\left(\int_{-1}^1\big(|\partial_zv|
+|v|\big)dz\right)dMdt\\
\leq & C\int_0^{t_0}\|\nabla_HV_\varepsilon\|_2\|V_\varepsilon\|_2^{\frac{1}{2}}(\|V_\varepsilon\|_2^{\frac{1}{2}}+\|\nabla_HV_\varepsilon\|_2^{\frac{1}{2}})\\
& \times\Big[\|\partial_zv\|_2^{\frac{1}{2}}(\|\partial_zv\|_2^{\frac{1}{2}}+\|\nabla_H\partial_zv\|_2^{\frac{1}{2}})+\|v\|_2^{\frac{1}{2}}(\|v\|_2^{\frac{1}{2}}+\|\nabla_Hv\|_2^{\frac{1}{2}})\Big]dt\\
\leq& \frac{1}{16}\|\nabla_HV_\varepsilon\|_{L^2(Q_{t_0})}^2
+C\int_0^{t_0}\|V_\varepsilon\|_2^2\Big[\|\partial_zv\|_2^{2}(\|\partial_zv\|_2^{2}+\|\nabla_H\partial_zv\|_2^{2})\\
& +\|v\|_2^{2}(\|v\|_2^{2}+\|\nabla_Hv\|_2^{2})+1\Big]dt\\
\leq&\frac{1}{16}\|\nabla_HV_\varepsilon\|_{L^2(Q_{t_0})}^2+C\int_0^{t_0}\|V_\varepsilon\|_2^2(1+\|\nabla_H\partial_zv\|_2^{2})dt,
\end{aligned}
\end{equation*}
where Proposition \ref{pro1} has been used. For $I_{42}$, it can be estimated in the same way as follows
\begin{equation*}
\begin{aligned}
I_{42}&=\int_0^{t_0}\int_\Omega\Big(\nabla_H\cdot V_\varepsilon V_\varepsilon\cdot v-W_\varepsilon V_\varepsilon\cdot\partial_zv\Big)d\Omega dt\\
&\leq \int_0^{t_0}\int_M\left(\int_{-1}^1|\nabla_HV_\varepsilon||V_\varepsilon|dz\right)\left(
\int_{-1}^1\big(|\partial_zv|+\frac{1}{2}|v|\big)dz\right)dMdt\\
&\qquad+\int_0^{t_0}\int_M\left(\int_{-1}^1|\nabla_HV_\varepsilon|dz\right)\left(
\int_{-1}^1|V_\varepsilon||\partial_zv|dz\right)dMdt\\
&\leq \frac{1}{16}\|\nabla_HV_\varepsilon\|_{L^2(Q_{t_0})}^2+C\int_0^{t_0}\|V_\varepsilon\|_2^2(1+\|\nabla_H\partial_zv\|_2^{2})dt.
\end{aligned}
\end{equation*}
Therefore, we have
\begin{equation*}
I_4\leq \frac{1}{8}\|\nabla_HV_\varepsilon\|_{L^2_{t_0}L^2}^2+C\int_0^{t_0}\|V_\varepsilon\|_2^2(1+\|\nabla_H\partial_zv\|_2^{2})dt.
\end{equation*}

Combining the above estimates of $I_1$, $I_2$, $I_3$, and $I_4$, by Proposition \ref{pro1}, one obtains
\begin{equation*}
\begin{aligned}
&f(t):=\|V_\varepsilon(t)\|_2^2+\varepsilon^2\|W_\varepsilon(t)\|_2^2\\
&\qquad+\int_0^{t}\Big(\|\nabla_HV_\varepsilon\|_2^2+\varepsilon^2\|\nabla_HW_\varepsilon\|_2^2+\varepsilon^{\alpha-2}\|\partial_zV_\varepsilon\|_2^2+\varepsilon^\alpha\|\partial_zW_\varepsilon\|_2^2\Big)ds\\
\leq& C\varepsilon^\beta+C\int_0^{t}\|V_\varepsilon\|_2^2(1+\|\nabla_H\partial_zv\|_2^2)ds=:F(t),
\end{aligned}
\end{equation*}
for a.e. $t\in[0,t^*]$. Therefore,
\begin{equation*}
\begin{aligned}
F'(t)&=C(1+\|\nabla_H\partial_zv\|_2^2)\|V_\varepsilon\|_2^2
\leq C(1+\|\nabla_H\partial_zv\|_2^2)f(t)\\
&\leq C(1+\|\nabla_H\partial_zv\|_2^2)F(t).
\end{aligned}
\end{equation*}
Then, by the Gronwall inequality and Proposition \ref{pro1}, we have
\begin{equation*}
\begin{aligned}
f(t)\leq F(t)\leq e^{C\int_0^{t^*}(1+\|\nabla_H\partial_zv\|_2^2)dt}F(0)
\leq C\varepsilon^\beta,
\end{aligned}
\end{equation*}
for a.e. $t\in[0,t^*]$, where $C$ depends only on $t^*$, $\|v_0\|_{H^1}$, $L_1$ and $L_2$. This proves (\ref{*}) and, thus, (i) holds.

(ii) Similar to (i), it suffices to show that

\begin{equation}
\begin{aligned}
&\sup_{0\leq t\leq T}\big(\|V_\varepsilon\|_2^2+\varepsilon^2\|W_\varepsilon\|_2^2\big)(t)\\
&\qquad+\int_0^{T}\Big(\|\nabla_HV_\varepsilon\|_2^2+\varepsilon^2\|\nabla_HW_\varepsilon\|_2^2+\varepsilon^{\alpha-2}\|\partial_zV_\varepsilon\|_2^2+\varepsilon^\alpha\|\partial_zW_\varepsilon\|_2^2\Big)dt\\
&\leq K(T)\varepsilon^\beta,\label{**}
\end{aligned}
\end{equation}
where $K(T)>0$ is a continuously increasing function determined by $\|v_0\|_{H^1}$, $\|\partial_zv_0\|_m$, $L_1$, $L_2$, and $t^*$. This can be proved exactly in the same way as (i), as in this case the a priori estimates used for proving (i) are valid up to any finite time $T$.
\end{proof}

\section{Proof of Theorem \ref{thm2}}

Suppose $v_0\in H^2(\Omega)$ with $\nabla_H\cdot\int_{-1}^1v_0 dz=0$. Then, by \eqref{9}, it has $u_0=(v_0,w_0)\in H^1(\Omega)$ and $\nabla\cdot u_0=0$. By the classical theory of Navier-Stokes equations (see \cite{Con} and \cite{Tem}), there is a unique local strong solution $(v_\varepsilon,w_\varepsilon)$ to \eqref{2}, subject to \eqref{3}--\eqref{5}. Denote by $T_\varepsilon^*$ the maximal existence time of $(v_\varepsilon,w_\varepsilon)$. Let $v$ be the global strong solution to \eqref{6} established in Proposition \ref{pro4}.

Here we still denote $U_\varepsilon=(V_\varepsilon,W_\varepsilon)$, and $V_\varepsilon=v_\varepsilon-v$, $W_\varepsilon=w_\varepsilon-w$. Since both $v$ and $(v_\varepsilon,w_\varepsilon)$ are strong solutions to \eqref{6} and \eqref{2}, respectively, one can check that $(V_\varepsilon,W_\varepsilon)$ satisfies
\begin{equation}
\begin{aligned}
\partial_tV_\varepsilon-\Delta_HV_\varepsilon-\varepsilon^{\alpha-2}\partial_z^2V_\varepsilon&+(U_\varepsilon\cdot\nabla)V_\varepsilon+\nabla_HP_\varepsilon\\
&+(U_\varepsilon\cdot\nabla)v+(u\cdot\nabla)V_\varepsilon=\varepsilon^{\alpha-2}\partial_z^2v,\label{18}
\end{aligned}
\end{equation}
\begin{equation}
\nabla_H\cdot V_\varepsilon+\partial_zW_\varepsilon=0,\label{19}
\end{equation}
\begin{equation}
\begin{aligned}
\varepsilon^2(\partial_tW_\varepsilon-\Delta_HW_\varepsilon&-\varepsilon^{\alpha-2}\partial_z^2W_\varepsilon+U_\varepsilon\cdot\nabla W_\varepsilon+U_\varepsilon\cdot\nabla w+u\cdot\nabla W_\varepsilon)\\
&+\partial_zP_\varepsilon=-\varepsilon^2(\partial_tw-\Delta_Hw-\varepsilon^{\alpha-2}\partial_z^2w+u\cdot\nabla w),\label{20}
\end{aligned}
\end{equation}
in $L^2(0,T_\varepsilon^*;L^2(\Omega))$, where $P_\varepsilon=p_\varepsilon-p$.

Since $v_0\in H^2(\Omega)$, it is clear that (\ref{**}) still holds  when $t\in[0,T_\varepsilon^*)$. In other words, the following holds:

\begin{equation}
\begin{aligned}
&\sup_{0\leq s\leq t}\big(\|V_\varepsilon\|_2^2+\varepsilon^2\|W_\varepsilon\|_2^2\big)(s)\\
&\qquad+\int_0^{t}\Big(\|\nabla_HV_\varepsilon\|_2^2+\varepsilon^2\|\nabla_HW_\varepsilon\|_2^2+\varepsilon^{\alpha-2}\|\partial_zV_\varepsilon\|_2^2+\varepsilon^\alpha\|\partial_zW_\varepsilon\|_2^2\Big)ds\\
&\leq K_1(t)\varepsilon^\beta,\label{***}
\end{aligned}
\end{equation}
for $t\in[0,T_\varepsilon^*)$, where $K_1(t):[0,\infty)\mapsto \mathbb{R}^+$ is a continuously increasing function determined by $\|v_0\|_{H^2}$, $L_1$ and $L_2$.

Besides the basic energy estimate stated in the above, we also have the first order energy estimate of $(V_\varepsilon,W_\varepsilon)$ in the following proposition.
\begin{proposition}
\label{pro5.1}
There exists a small constant $\sigma>0$ depending only on $L_1$ and $L_2$, such that the following inequality holds
\begin{equation*}
\begin{aligned}
&\sup_{0\leq s\leq t}\big(\|\nabla V_\varepsilon\|_2^2+\varepsilon^2\|\nabla W_\varepsilon\|_2^2\big)(s)\\
&\qquad+\int_0^{t}\Big(\|\nabla\nabla_HV_\varepsilon\|_2^2+\varepsilon^2\|\nabla\nabla_HW_\varepsilon\|_2^2+\varepsilon^{\alpha-2}\|\nabla\partial_zV_\varepsilon\|_2^2+\varepsilon^\alpha\|\nabla\partial_zW_\varepsilon\|_2^2\Big)ds\\
&\leq K_2(t)\varepsilon^\beta,
\end{aligned}
\end{equation*}
for any $t\in[0,T_\varepsilon^*)$, as long as
\begin{equation*}
\sup_{0\leq s\leq t}\big(\|\nabla V_\varepsilon\|_2^2+\varepsilon^2\|\nabla W_\varepsilon\|_2^2\big)(s)\leq\sigma^2,
\end{equation*}
where $K_2(t):[0,\infty)\mapsto\mathbb{R}^+$ is a continuously increasing function determined by $\|v_0\|_{H^2}$, $L_1$ and $L_2$.
\end{proposition}

\begin{proof}
Since \eqref{18} holds in $L^2((0,T_\varepsilon^*)\times \Omega)$ and $-\Delta V_{\varepsilon}\in L^2((0,T_\varepsilon^*)\times \Omega)$, one can multiply \eqref{18} with $-\Delta V_\varepsilon$, integrating over $\Omega$, and by integration by parts, to get
\begin{equation*}
\begin{aligned}
&\frac{1}{2}\frac{d}{dt}\|\nabla V_\varepsilon\|_2^2+\|\nabla\nabla_HV_\varepsilon\|_2^2+\varepsilon^{\alpha-2}\|\nabla\partial_zV_\varepsilon\|_2^2+\int_\Omega \nabla_HP_\varepsilon\cdot\Delta V_\varepsilon d\Omega\\
=&\int_\Omega\Big[(U_\varepsilon\cdot\nabla)V_\varepsilon+(U_\varepsilon\cdot\nabla)v+(u\cdot\nabla)V_\varepsilon\Big]\cdot\Delta V_\varepsilon~d\Omega-\int_\Omega\varepsilon^{\alpha-2}\partial_z^2v\cdot\Delta V_\varepsilon~d\Omega.
\end{aligned}
\end{equation*}
We estimate the terms on the right hand side of the above equality as follows. By Lemma \ref{lem1} and using $|f(x,y,z)|\leq\frac{1}{2}\int_{-1}^1|f|dz+\int_{-1}^1|\partial_zf|dz$, one deduces
\begin{align*}
&\int_\Omega(U_\varepsilon\cdot\nabla)V_\varepsilon\cdot\Delta V_\varepsilon~d\Omega\\
=&\int_\Omega\Big((V_\varepsilon\cdot\nabla_H)V_\varepsilon\cdot\Delta_H V_\varepsilon+W_\varepsilon\partial_zV_\varepsilon\cdot\Delta_H V_\varepsilon\Big)d\Omega\\
&+\int_\Omega\Big((V_\varepsilon\cdot\nabla_H)V_\varepsilon\cdot\partial_z^2V_\varepsilon+W_\varepsilon\partial_zV_\varepsilon\cdot\partial_z^2V_\varepsilon\Big)d\Omega\\
=&\int_\Omega\Big((V_\varepsilon\cdot\nabla_H)V_\varepsilon\cdot\Delta_H V_\varepsilon+W_\varepsilon\partial_zV_\varepsilon\cdot\Delta_H V_\varepsilon\Big)d\Omega\\
&-\int_\Omega\Big((\partial_zV_\varepsilon\cdot\nabla_H)V_\varepsilon\cdot\partial_zV_\varepsilon-\nabla_H\cdot V_\varepsilon|\partial_zV_\varepsilon|^2\Big)d\Omega\\
\leq&\int_M\left(\int_{-1}^1\big(|\partial_zV_\varepsilon|+|V_\varepsilon|\big)dz\right)\left(
\int_{-1}^1|\nabla_HV_\varepsilon||\Delta_HV_\varepsilon|dz\right)dM\\
&+\int_M\left(\int_{-1}^1|\nabla_HV_\varepsilon|dz\right)\left(\int_{-1}^1|\partial_zV_\varepsilon||\Delta_HV_\varepsilon|dz\right)dM\\
&+2\int_M\left(\int_{-1}^1\big(|\nabla_H\partial_zV_\varepsilon|+\frac{1}{2}|\nabla_HV_\varepsilon|\big)dz\right)\left(
\int_{-1}^1|\partial_zV_\varepsilon|^2dz\right)dM\\
\leq&C\|\Delta_HV_\varepsilon\|_2\|\nabla_HV_\varepsilon\|_2^{\frac{1}{2}}\big(\|\nabla_HV_\varepsilon\|_2^{\frac{1}{2}}+\|\Delta_HV_\varepsilon\|_2^{\frac{1}{2}}\big)\\
&\qquad\times \Big[\|\partial_zV_\varepsilon\|_2^{\frac{1}{2}}\big(\|\partial_zV_\varepsilon\|_2^{\frac{1}{2}}+\|\nabla_H\partial_zV_\varepsilon\|_2^{\frac{1}{2}}\big)+\|V_\varepsilon\|_2^{\frac{1}{2}}\big(\|V_\varepsilon\|_2^{\frac{1}{2}}+\|\nabla_HV_\varepsilon\|_2^{\frac{1}{2}}\big)\Big]\\
&+C\big(\|\nabla_H\partial_zV_\varepsilon\|_2+\|\nabla_HV_\varepsilon\|_2\big)\|\partial_zV_\varepsilon\|_2\big(\|\partial_zV_\varepsilon\|_2+\|\nabla_H\partial_zV_\varepsilon\|_2\big)\\
\leq&\frac{1}{16}\|\nabla\nabla_HV_\varepsilon\|_2^2+C\big(\|\nabla_HV_\varepsilon\|_2^4+\|\nabla_HV_\varepsilon\|_2^2\|\Delta_HV_\varepsilon\|_2^2\big)+C\|\nabla V_\varepsilon\|_2^3\\
&+C\big(\|\partial_zV_\varepsilon\|_2^4+\|\partial_zV_\varepsilon\|_2^2\|\nabla_H\partial_zV_\varepsilon\|_2^2+\|V_\varepsilon\|_2^4+\|V_\varepsilon\|_2^2\|\nabla_HV_\varepsilon\|_2^2\big)
\end{align*}
Integrating by parts, using $|f(x,y,z)|\leq\frac{1}{2}\int_{-1}^1|f|dz+\int_{-1}^1|\partial_zf|dz$ and applying Lemma \ref{lem1}, one deduces by the Young inequality that
\begin{align*}
&\int_\Omega(U_\varepsilon\cdot\nabla)v\cdot\Delta V_\varepsilon~d\Omega\\
= &\int_\Omega\Big((V_\varepsilon\cdot\nabla_H)v\cdot\Delta_HV_\varepsilon-(\partial_zV_\varepsilon\cdot\nabla_H)v\cdot\partial_zV_\varepsilon-(V_\varepsilon\cdot\nabla_H)\partial_zv\cdot\partial_zV_\varepsilon\Big)d\Omega\\
&+\int_\Omega \Big(W_\varepsilon\partial_zv\cdot\Delta_HV_\varepsilon+\nabla_H\cdot V_\varepsilon\partial_zv\cdot\partial_zV_\varepsilon-W_\varepsilon\partial_z^2v\cdot\partial_zV_\varepsilon\Big)d\Omega\\
\leq& \int_M\left(\int_{-1}^1\big(|V_\varepsilon|+|\partial_zV_\varepsilon|\big)dz\right)\left(
\int_{-1}^1|\nabla_Hv||\Delta_HV_\varepsilon|dz\right)dM\\
&+\int_M\left(\int_{-1}^1|\partial_zV_\varepsilon|^2dz\right)\left(\int_{-1}^1\big(|\nabla_Hv|+|\nabla_H\partial_zv|\big)dz\right)dM\\
&+\int_M\left(\int_{-1}^1\big(|V_\varepsilon|+|\partial_zV_\varepsilon|\big)dz\right)\left(
\int_{-1}^1|\nabla_H\partial_zv||\partial_zV_\varepsilon| dz\right)dM\\
&+\int_M\left(\int_{-1}^1|\nabla_HV_\varepsilon|dz\right)\left(\int_{-1}^1|\partial_zv||\Delta_HV_\varepsilon|dz\right)dM\\
&+\int_M\left(\int_{-1}^1|\nabla_HV_\varepsilon||\partial_zV_\varepsilon|dz\right)\left(\int_{-1}^1|\partial_z^2v|dz\right)dM\\
&+\int_M\left(\int_{-1}^1|\nabla_HV_\varepsilon|dz\right)\left(\int_{-1}^1|\partial_z^2v||\partial_zV_\varepsilon|dz\right)dM\\
\leq& C\|\Delta_HV_\varepsilon\|_2\|\nabla_Hv\|_2^{\frac{1}{2}}\|\Delta_Hv\|_2^{\frac{1}{2}}\\
&\qquad\times \Big[\|\partial_zV_\varepsilon\|_2^{\frac{1}{2}}\big(\|\partial_zV_\varepsilon\|_2^{\frac{1}{2}}+\|\nabla_H\partial_zV_\varepsilon\|_2^{\frac{1}{2}}\big)+\|V_\varepsilon\|_2^{\frac{1}{2}}\big(\|V_\varepsilon\|_2^{\frac{1}{2}}+\|\nabla_HV_\varepsilon\|_2^{\frac{1}{2}}\big)\Big]\\
&+C(\|\nabla_Hv\|_2+\|\nabla_H\partial_zv\|_2)\\
&\qquad\times\Big[\|\partial_zV_\varepsilon\|_2(\|\partial_zV_\varepsilon\|_2+\|\nabla_H\partial_zV_\varepsilon\|_2)+\|V_\varepsilon\|_2(\|V_\varepsilon\|_2+\|\nabla_HV_\varepsilon\|_2)\Big]\\
&+C\|\Delta_HV_\varepsilon\|_2\|\partial_zv\|_2^{\frac{1}{2}}(\|\partial_zv\|_2^{\frac{1}{2}}+\|\nabla_H\partial_zv\|_2^{\frac{1}{2}})\|\nabla_HV_\varepsilon\|_2^{\frac{1}{2}}\|\Delta_HV_\varepsilon\|_2^{\frac{1}{2}}\\
&+C\|\partial_z^2v\|_2\|\nabla_HV_\varepsilon\|_2^{\frac{1}{2}}\|\Delta_HV_\varepsilon\|_2^{\frac{1}{2}}\|\partial_zV_\varepsilon\|_2^{\frac{1}{2}}(\|\partial_zV_\varepsilon\|_2^{\frac{1}{2}}+\|\nabla_H\partial_zV_\varepsilon\|_2^{\frac{1}{2}})\\
\leq& \frac{1}{16}\|\nabla\nabla_HV_\varepsilon\|_2^2+C(1+\|v\|_{H^1}^2)(1+\|v\|_{H^2}^2)\|V_\varepsilon\|_{H^1}^2
\end{align*}
and
\begin{align*}
&\int_\Omega(u\cdot\nabla)V_\varepsilon\cdot\Delta V_\varepsilon~d\Omega\\
=&\int_\Omega\Big( (u\cdot\nabla)V_\varepsilon\Delta_HV_\varepsilon-(\partial_zu\cdot\nabla)V_\varepsilon\cdot\partial_zV_\varepsilon-(u\cdot\nabla)\partial_zV_\varepsilon\cdot\partial_zV_\varepsilon\Big)d\Omega\\
=&\int_\Omega \Big( (v\cdot\nabla_H)V_\varepsilon\cdot\Delta_HV_\varepsilon+w\partial_zV_\varepsilon\cdot\Delta_HV_\varepsilon-(\partial_zv\cdot\nabla_H)V_\varepsilon\cdot\partial_zV_\varepsilon+\nabla_H\cdot v|\partial_zV_\varepsilon|^2\Big)d\Omega\\
\leq&\int_M\left(\int_{-1}^1\big(|v|+|\partial_zv|\big)dz\right)\left(
\int_{-1}^1|\nabla_HV_\varepsilon||\Delta_HV_\varepsilon|dz\right)dM\\
&+\int_M\left(\int_{-1}^1|\nabla_Hv|dz\right)\left(\int_{-1}^1|\partial_zV_\varepsilon||\Delta_HV_\varepsilon|dz\right)dM\\
&+\int_M\left(\int_{-1}^1|\partial_z^2v|dz\right)\left(\int_{-1}^1|\nabla_HV_\varepsilon||\partial_zV_\varepsilon|dz\right)dM\\
&+2\int_M\left(\int_{-1}^1\big(|\nabla_Hv|+|\nabla_H\partial_zv|\big)dz\right)\left(\int_{-1}^1|\partial_zV_\varepsilon|^2dz
\right)dM\\
\leq&C\|\Delta_HV_\varepsilon\|_2\|\nabla_HV_\varepsilon\|_2^{\frac{1}{2}}\|\Delta_HV_\varepsilon\|_2^{\frac{1}{2}}\\
&\qquad\times\Big[\|v\|_2^{\frac{1}{2}}(\|v\|_2^{\frac{1}{2}}+\|\nabla_Hv\|_2^{\frac{1}{2}})+\|\partial_zv\|_2^{\frac{1}{2}}(\|\partial_zv\|_2^{\frac{1}{2}}+\|\nabla_H\partial_zv\|_2^{\frac{1}{2}})\Big]\\
&+C\|\Delta_HV_\varepsilon\|_2\|\nabla_Hv\|_2^{\frac{1}{2}}\|\Delta_Hv\|_2^{\frac{1}{2}}\|\partial_zV_\varepsilon\|_2^{\frac{1}{2}}(\|\partial_zV_\varepsilon\|_2^{\frac{1}{2}}+\|\nabla_H\partial_zV_\varepsilon\|_2^{\frac{1}{2}})\\
&+C\|\partial_z^2v\|_2\|\nabla_HV_\varepsilon\|_2^{\frac{1}{2}}\|\Delta_HV_\varepsilon\|_2^{\frac{1}{2}}\\
&\qquad\times\|\partial_zV_\varepsilon\|_2^{\frac{1}{2}}(\|\partial_zV_\varepsilon\|_2^{\frac{1}{2}}+\|\nabla_H\partial_zV_\varepsilon\|_2^{\frac{1}{2}})\\
&+C(\|\nabla_Hv\|_2+\|\nabla_H\partial_zv\|_2)\|\partial_zV_\varepsilon\|_2(\|\partial_zV_\varepsilon\|_2+\|\nabla_H\partial_zV_\varepsilon\|_2)\\
\leq&\frac{1}{16}\|\nabla\nabla_HV_\varepsilon\|_2^2+C(\|\nabla V_\varepsilon\|_2^2+\|V_\varepsilon\|_2^2)(\|v\|_{H^1}^2+1)(\|v\|_{H^2}^2+1),
\end{align*}
where the Poincar$\acute{e}$ inequality $\|\nabla_Hf\|_2\leq C\|\nabla_H^2f\|_2$ has been used in several places. The Cauchy inequality yields
\begin{align*}
&\int_\Omega\varepsilon^{\alpha-2}\partial_z^2v\cdot\Delta V_\varepsilon~d\Omega\\
\leq&\varepsilon^{\alpha-2}\|\partial_z^2v\|_2(\|\Delta_HV_\varepsilon\|_2+\|\partial_z^2V_\varepsilon\|_2)\\
\leq&\frac{1}{16}\Big(\|\Delta_HV_\varepsilon\|_2^2+\varepsilon^{\alpha-2}\|\partial_z^2V_\varepsilon\|_2^2\Big)+C\big(\varepsilon^{\alpha-2}+\varepsilon^{2(\alpha-2)}\big)\|\partial_z^2v\|_2^2.
\end{align*}
Combining all the above estimates and applying Proposition \ref{pro4}, one deduces
\begin{equation}
\begin{aligned}
&\frac{1}{2}\frac{d}{dt}\|\nabla V_\varepsilon\|_2^2+\frac{3}{4}\Big(\|\nabla\nabla_HV_\varepsilon\|_2^2+\varepsilon^{\alpha-2}\|\nabla\partial_zV_\varepsilon\|_2^2\Big)+\int_\Omega \nabla_HP_\varepsilon\cdot\Delta V_\varepsilon d\Omega\\
\leq&C\varepsilon^{\alpha-2}G(t)+C(G^2(t)+1)\|\nabla V_\varepsilon\|_2^2
+C\|\nabla V_\varepsilon\|_2^2\|\nabla\nabla_HV_\varepsilon\|_2^2+C\|V_\varepsilon\|_{H^1}^4.\label{21}
\end{aligned}
\end{equation}

Recall that \eqref{20} holds in $L^2((0,T_\varepsilon^*)\times \Omega)$ and $-\Delta W_{\varepsilon}\in L^2((0,T_\varepsilon^*)\times \Omega)$. Multiplying \eqref{20} with $-\Delta W_\varepsilon$ and integrating over $\Omega$, one has
\begin{align*}
&\frac{\varepsilon^2}{2}\frac{d}{dt}\|\nabla W_\varepsilon\|_2^2+\varepsilon^2\|\nabla\nabla_HW_\varepsilon\|_2^2+\varepsilon^\alpha\|\partial_z\nabla W_\varepsilon\|_2^2+\int_\Omega \partial_zP_\varepsilon\Delta W_\varepsilon d\Omega\\
=&\varepsilon^2\int_\Omega\Big( U_\varepsilon\cdot\nabla W_\varepsilon\Delta W_\varepsilon+U_\varepsilon\cdot\nabla w\Delta W_\varepsilon+u\cdot\nabla W_\varepsilon\Delta W_\varepsilon\Big)d\Omega+\varepsilon^2\int_\Omega u\cdot\nabla w\Delta W_\varepsilon~d\Omega\\
&+\varepsilon^2\int_\Omega\Big(\partial_t w\Delta W_\varepsilon-\Delta_Hw\Delta W_\varepsilon-\varepsilon^{\alpha-2}\partial_z^2w\Delta W_\varepsilon\Big)d\Omega,
\end{align*}
Using $|f(x,y,z)|\leq \int_{-1}^1\big(|\partial_zf|+\frac{1}{2}|f|\big)dz$, applying Lemma \ref{lem1}, and by the Young inequality, one deduces
\begin{align*}
&\varepsilon^2\int_\Omega U_\varepsilon\cdot\nabla W_\varepsilon\Delta W_\varepsilon~d\Omega\\
=&\varepsilon^2\int_\Omega \Big(V_\varepsilon\cdot\nabla_HW_\varepsilon\Delta_HW_\varepsilon-V_\varepsilon\cdot\nabla_H\partial_zW_\varepsilon\partial_zW_\varepsilon-\partial_zV_\varepsilon\cdot\nabla_HW_\varepsilon\partial_zW_\varepsilon\\
&\qquad\qquad+W_\varepsilon\cdot\partial_zW_\varepsilon\Delta_HW_\varepsilon-\frac{1}{2}\partial_zW_\varepsilon|\partial_zW_\varepsilon|^2\Big)d\Omega\\
=&\varepsilon^2\int_\Omega\Big( V_\varepsilon\cdot\nabla_HW_\varepsilon\Delta_HW_\varepsilon-2V_\varepsilon\cdot\nabla_H\partial_zW_\varepsilon\partial_zW_\varepsilon
\Big)d\Omega\\
&-\varepsilon^2\int_\Omega \partial_zV_\varepsilon\cdot\left(\nabla_H\int_{-1}^z\nabla_H\cdot V_\varepsilon dz'\right)\nabla_H\cdot V_\varepsilon d\Omega\\
&+\varepsilon^2\int_\Omega\left(\int_{-1}^z\nabla_H\cdot V_\varepsilon dz'\right)\nabla_H\cdot V_\varepsilon\Delta_HW_\varepsilon d\Omega\\
\leq&C\varepsilon^2\int_M\left(\int_{-1}^1\big(|V_\varepsilon|+|\partial_zV_\varepsilon|\big)dz\right)
\left(\int_{-1}^1|\nabla W_\varepsilon||\nabla\nabla_HW_\varepsilon|dz\right)dM\\
&+C\varepsilon^2\int_M\left(\int_{-1}^1|\nabla_H^2V_\varepsilon|dz\right)\left(
\int_{-1}^1|\partial_zV_\varepsilon||\nabla_HV_\varepsilon|dz\right)dM\\
&+C\varepsilon^2\int_M\left(\int_{-1}^1|\nabla_HV_\varepsilon|dz\right)\left(
\int_{-1}^1|\nabla_HV_\varepsilon||\Delta_HW_\varepsilon|dz\right)dM\\
\leq&C\varepsilon^2\|\nabla_H\nabla W_\varepsilon\|_2\|\nabla W_\varepsilon\|_2^{\frac{1}{2}}(\|\nabla W_\varepsilon\|_2^{\frac{1}{2}}+\nabla_H\nabla W_\varepsilon\|_2^{\frac{1}{2}})\\
&\qquad\times\Big[\|\partial_zV_\varepsilon\|_2^{\frac{1}{2}}(\|\partial_zV_\varepsilon\|_2^{\frac{1}{2}}+\|\nabla_H\partial_zV_\varepsilon\|_2^{\frac{1}{2}})+\|V_\varepsilon\|_2^{\frac{1}{2}}(\|V_\varepsilon\|_2^{\frac{1}{2}}+\|\nabla_HV_\varepsilon\|_2^{\frac{1}{2}})\Big]\\
&+C\varepsilon^2\|\nabla_H^2V_\varepsilon\|_2\|\partial_zV_\varepsilon\|_2^{\frac{1}{2}}(\|\partial_zV_\varepsilon\|_2+\|\nabla_H\partial_zV_\varepsilon\|_2)^{\frac{1}{2}}\|\nabla_HV_\varepsilon\|_2^{\frac{1}{2}}\|\nabla_H^2V_\varepsilon\|_2^{\frac{1}{2}}\\
&+C\varepsilon^2\|\nabla_HV_\varepsilon\|_2\|\nabla_H^2V_\varepsilon\|_2\|\Delta_H W_\varepsilon\|_2\\
\leq&\frac{\varepsilon^2}{16}\|\nabla\nabla_HW_\varepsilon\|_2^2+\frac{1}{32}\|\nabla_H^2V_\varepsilon\|_2^2+C\|\nabla V_\varepsilon\|_2^2(\|\nabla V_\varepsilon\|_2^2+\|\nabla\nabla_HV_\varepsilon\|_2^2)\\
&+C\varepsilon^2\|\nabla W_\varepsilon\|_2^2(\varepsilon^2\|\nabla W_\varepsilon\|_2^2+\varepsilon^2\|\nabla\nabla_HW_\varepsilon\|_2^2)
+C\|V_\varepsilon\|_2^2(\|V_\varepsilon\|_2^2+\|\nabla_HV_\varepsilon\|_2^2),
\end{align*}
where the incompressibility condition \eqref{19} and the Poincar\'e inequality have been used. Similarly and using further $|W_\varepsilon|=\Big|\int_{-1}^z\partial_zW_\varepsilon(x,y,z')dz'\Big|\leq\int_{-1}^1|\partial_zW_\varepsilon|dz'$ as $W_\varepsilon|_{z=-1}=0$, one deduces
\begin{align*}
&\varepsilon^2\int_\Omega \big(U_\varepsilon\cdot\nabla w\big)\Delta W_\varepsilon~d\Omega\\
=&\varepsilon^2\int_\Omega \Big( \big(V_\varepsilon\cdot\nabla_Hw\big)\Delta_HW_\varepsilon-\big(\partial_zV_\varepsilon\cdot\nabla_Hw\big)\partial_zW_\varepsilon-\big(V_\varepsilon\cdot\nabla_H\partial_zw\big)\partial_zW_\varepsilon\\
&\qquad\qquad+W_\varepsilon\partial_zw\Delta_HW_\varepsilon-|\partial_zW_\varepsilon|^2\partial_zw-W_\varepsilon\partial_z^2w\partial_zW_\varepsilon\Big)d\Omega\\
=&\varepsilon^2\int_\Omega\left[-\left(V_\varepsilon\cdot\nabla_H\int_{-1}^z\nabla_H\cdot vdz'\right)\Delta_HW_\varepsilon+\left(\partial_zV_\varepsilon\cdot\nabla_H\int_{-1}^z\nabla_H\cdot vdz'\right)\partial_zW_\varepsilon\right]d\Omega\\
&+\varepsilon^2\int_\Omega\left[\left(V_\varepsilon\cdot\nabla_H(\nabla_H\cdot v)\right)\partial_zW_\varepsilon+\left(\int_{-1}^z\nabla_H\cdot V_\varepsilon dz'\right)\nabla_H\cdot v\Delta_HW_\varepsilon\right]d\Omega\\
&+\varepsilon^2\int_\Omega\Big(|\partial_zW_\varepsilon|^2\nabla_H\cdot v+W_\varepsilon(\nabla_H\cdot\partial_zv)\partial_zW_\varepsilon\Big)d\Omega\\
\leq&C\varepsilon^2\int_M\left(\int_{-1}^1|\nabla_H^2v|dz\right)\left(\int_{-1}^1|V_\varepsilon||\Delta_HW_\varepsilon|dz\right)dM\\
&+C\varepsilon^2\int_M\left(\int_{-1}^1|\nabla_H^2v|dz\right)\left(
\int_{-1}^1|\partial_zV_\varepsilon||\partial_zW_\varepsilon|dz\right)dM\\
&+C\varepsilon^2\int_M\left(\int_{-1}^1\big(|V_\varepsilon|+|\partial_zV_\varepsilon|\big)dz\right)\left(
\int_{-1}^1|\nabla_H^2v||\partial_zW_\varepsilon|dz\right)dM\\
&+C\varepsilon^2\int_M\left(\int_{-1}^1|\nabla_HV_\varepsilon|dz\right)\left(
\int_{-1}^1|\nabla_Hv||\Delta_HW_\varepsilon|dz\right)dM\\
&+C\varepsilon^2\int_M\left(\int_{-1}^1|\partial_zW_\varepsilon|^2dz\right)\left(
\int_{-1}^1\big(|\nabla_Hv|+|\nabla_H\partial_zv|\big)dz\right)dM\\
&+C\varepsilon^2\int_M\left(\int_{-1}^1|\partial_zW_\varepsilon|dz\right)\left(
\int_{-1}^1|\nabla_H\partial_zv||\partial_zW_\varepsilon|dz\right)dM\\
\leq&C\varepsilon^2\|\Delta_HW_\varepsilon\|_2\|V_\varepsilon\|_2^{\frac{1}{2}}(\|V_\varepsilon\|_2^{\frac{1}{2}}+\|\nabla_HV_\varepsilon\|_2^{\frac{1}{2}})\|\Delta_Hv\|_2^{\frac{1}{2}}\|\nabla_H\Delta_Hv\|_2^{\frac{1}{2}}\\
&+C\varepsilon^2\|\Delta_Hv\|_2\|\partial_zW_\varepsilon\|_2^{\frac{1}{2}}(\|\partial_zW_\varepsilon\|_2^{\frac{1}{2}}+\|\nabla_H\partial_zW_\varepsilon\|_2^{\frac{1}{2}})\\
&\qquad\times\Big[\|\partial_zV_\varepsilon\|_2^{\frac{1}{2}}(\|\partial_zV_\varepsilon\|_2^{\frac{1}{2}}+\|\nabla_H\partial_zV_\varepsilon\|_2^{\frac{1}{2}})+\|V_\varepsilon\|_2^{\frac{1}{2}}(\|V_\varepsilon\|_2^{\frac{1}{2}}+\|\nabla_HV_\varepsilon\|_2^{\frac{1}{2}})\Big]\\
&+C\varepsilon^2\|\Delta_HW_\varepsilon\|_2\|\nabla_Hv\|_2^{\frac{1}{2}}\|\Delta_Hv\|_2^{\frac{1}{2}}\|\nabla_HV_\varepsilon\|_2^{\frac{1}{2}}\|\Delta_HV_\varepsilon\|_2^{\frac{1}{2}}\\
&+C\varepsilon^2(\|\nabla_H\partial_zv\|_2+\|\nabla_Hv\|_2)\|\partial_zW_\varepsilon\|_2(\|\partial_zW_\varepsilon\|_2+\|\nabla_H\partial_zW_\varepsilon\|_2)\\
\leq&\frac{\varepsilon^2}{16}\|\nabla\nabla_HW_\varepsilon\|_2^2+\frac{1}{32}\|\nabla\nabla_HV_\varepsilon\|_2^2+C\|\nabla V_\varepsilon\|_2^2(\|v\|_{H^2}^4+1)\\
&+C(\|V_\varepsilon\|_2^2+\varepsilon^2\|\nabla_HW_\varepsilon\|_2^2)(\|v\|_{H^2}^4+\|v\|_{H^2}^2\|\nabla_Hv\|_{H^2}^2+1)\\
&+C\varepsilon^2\| W_\varepsilon\|_2^2(\|v\|_{H^2}^4+\|v\|_{H^2}^2\|\nabla_Hv\|_{H^2}^2+1).
\end{align*}
The other nonlinear terms can be estimated in the same way, by using Lemma \ref{lem1}, the Poincar$\acute{e}$ inequality and $|f(x,y,z)|\leq\frac{1}{2}\int_{-1}^1|f|dz+\int_{-1}^1|\partial_zf|dz$ as follows. In fact, one deduces
\begin{align*}
&\varepsilon^2\int_\Omega u\cdot\nabla W_\varepsilon\Delta W_\varepsilon~d\Omega\\
=&\varepsilon^2\int_\Omega\Big(v\cdot\nabla_HW_\varepsilon\Delta_HW_\varepsilon-v\cdot\nabla_H\partial_zW_\varepsilon\partial_zW_\varepsilon-\partial_zv\cdot\nabla_HW_\varepsilon\partial_zW_\varepsilon\\
&\qquad\qquad+w\partial_zW_\varepsilon\Delta_HW_\varepsilon-\frac{1}{2}\partial_zw|\partial_zW_\varepsilon|^2\Big)d\Omega\\
=&\varepsilon^2\int_\Omega\Big(v\cdot\nabla_HW_\varepsilon\Delta_HW_\varepsilon+\nabla_H\cdot v|\partial_zW_\varepsilon|^2\Big)d\Omega\\
&+\varepsilon^2\int_\Omega\left[\left(\partial_zv\cdot\nabla_H\int_{-1}^z\nabla_H\cdot V_\varepsilon dz'\right)\partial_zW_\varepsilon-\left(\int_{-1}^z\nabla_H\cdot vdz'\right)\partial_zW_\varepsilon\Delta_HW_\varepsilon\right]d\Omega\\
\leq&C\varepsilon^2\int_M\left(\int_{-1}^1\big(|v|+|\partial_zv|\big)dz\right)\left(
\int_{-1}^1|\nabla_HW_\varepsilon||\Delta_HW_\varepsilon|dz\right)dM\\
&+C\varepsilon^2\int_M\left(\int_{-1}^1\big(|\nabla_Hv|+|\nabla_H\partial_zv|\big)dz\right)\left(
\int_{-1}^1|\partial_zW_\varepsilon|^2dz\right)dM\\
&+C\varepsilon^2\int_M\left(\int_{-1}^1|\nabla_H^2V_\varepsilon|dz\right)\left(
\int_{-1}^1|\partial_zv||\partial_zW_\varepsilon|dz\right)dM\\
&+C\varepsilon^2\int_M\left(\int_{-1}^1|\nabla_Hv|dz\right)\left(
\int_{-1}^1|\partial_zW_\varepsilon||\Delta_HW_\varepsilon|dz\right)dM\\
\leq&C\varepsilon^2\|\Delta_HW_\varepsilon\|_2\|\nabla_HW_\varepsilon\|_2^{\frac{1}{2}}\|\Delta_HW_\varepsilon\|_2^{\frac{1}{2}}\\
&\qquad\times\Big[\|\partial_zv\|_2^{\frac{1}{2}}(\|\partial_zv\|_2^{\frac{1}{2}}+\|\nabla_H\partial_zv\|_2^{\frac{1}{2}})+\|v\|_2^{\frac{1}{2}}(\|v\|_2^{\frac{1}{2}}+\|\nabla_Hv\|_2^{\frac{1}{2}})\Big]\\
&+C\varepsilon^2(\|\nabla_H\partial_zv\|_2+\|\nabla_Hv\|_2)\|\partial_zW_\varepsilon\|_2(\|\partial_zW_\varepsilon\|_2+\|\nabla_H\partial_zW_\varepsilon\|_2)\\
&+C\varepsilon^2\|\Delta_HV_\varepsilon\|_2\|\partial_zv\|_2^{\frac{1}{2}}(\|\partial_zv\|_2^{\frac{1}{2}}+\|\nabla_H\partial_zv\|_2^{\frac{1}{2}})\\
&\qquad\qquad\times\|\partial_zW_\varepsilon\|_2^{\frac{1}{2}}(\|\partial_zW_\varepsilon\|_2^{\frac{1}{2}}+\|\nabla_H\partial_zW_\varepsilon\|_2^{\frac{1}{2}})\\
&+C\varepsilon^2\|\Delta_HW_\varepsilon\|_2\|\nabla_Hv\|_2^{\frac{1}{2}}\|\Delta_Hv\|_2^{\frac{1}{2}}\\
&\qquad\qquad\times\|\partial_zW_\varepsilon\|_2^{\frac{1}{2}}(\|\partial_zW_\varepsilon\|_2^{\frac{1}{2}}+\|\nabla_H\partial_zW_\varepsilon\|_2^{\frac{1}{2}})\\
\leq&\frac{\varepsilon^2}{16}\|\nabla\nabla_HW_\varepsilon\|_2^2+\frac{1}{32}\|\nabla\nabla_HV_\varepsilon\|_2^2+C\varepsilon^2\|\nabla W_\varepsilon\|_2^2(\|v\|_{H^2}^4+1),
\end{align*}
and
\begin{align*}
&\varepsilon^2\int_\Omega \big(u\cdot\nabla w\big)\Delta W_\varepsilon~d\Omega\\
=&\varepsilon^2\int_\Omega\Big(\big(v\cdot\nabla_Hw\big)\Delta_HW_\varepsilon-\big(\partial_zv\cdot\nabla_Hw\big)\partial_zW_\varepsilon-\big(v\cdot\nabla_H\partial_zw\big)\partial_zW_\varepsilon\\
&\qquad\qquad+w\partial_zw\Delta_HW_\varepsilon-|\partial_zw|^2\partial_zW_\varepsilon-w\partial_z^2w\partial_zW_\varepsilon\Big)d\Omega\\
=&\varepsilon^2\int_\Omega\left[-\left(v\cdot\nabla_H\int_{-1}^z\nabla_H\cdot vdz'\right)\Delta_HW_\varepsilon+\left(\partial_zv\cdot\nabla_H\int_{-1}^z\nabla_H\cdot vdz'\right)\partial_zW_\varepsilon\right]d\Omega\\
&+\varepsilon^2\int_\Omega\left[\Big(v\cdot\nabla_H(\nabla_H\cdot v)\Big)\partial_zW_\varepsilon+\left(\int_{-1}^z\nabla_H\cdot vdz'\right)\nabla_H\cdot v\Delta_HW_\varepsilon\right]d\Omega\\
&+\varepsilon^2\int_\Omega\left[|\nabla_H\cdot v|^2\nabla_H\cdot V_\varepsilon-\left(\int_{-1}^z\nabla_H\cdot vdz'\right)\nabla_H\cdot\partial_zv\partial_zW_\varepsilon\right]d\Omega\\
\leq&C\varepsilon^2\int_M\left(\int_{-1}^1|\nabla_H^2v|dz\right)\left(\int_{-1}^1|v||\Delta_HW_\varepsilon|dz\right)dM\\
&+C\varepsilon^2\int_M\left(\int_{-1}^1|\nabla_H^2v|dz\right)\left(\int_{-1}^1|\partial_zv||\partial_zW_\varepsilon|dz\right)dM\\
&+C\varepsilon^2\int_M\left(\int_{-1}^1\big(|v|+|\partial_zv|\big)dz\right)\left(
\int_{-1}^1|\nabla_H^2v||\partial_zW_\varepsilon|dz\right)dM\\
&+C\varepsilon^2\int_M\left(\int_{-1}^1|\nabla_Hv|dz\right)\left(\int_{-1}^1|\nabla_Hv||\Delta_HW_\varepsilon|dz\right)dM\\
&+C\varepsilon^2\int_M\left(\int_{-1}^1\big(|\nabla_HV_\varepsilon|+|\nabla_H\partial_zV_\varepsilon|\big)dz\right)
\left(\int_{-1}^1|\nabla_Hv|^2dz\right)dM\\
&+C\varepsilon^2\int_M\left(\int_{-1}^1|\nabla_Hv|dz\right)\left(\int_{-1}^1|\nabla_H\partial_zv||\partial_zW_\varepsilon|dz\right)dM\\
\leq&C\varepsilon^2\|\Delta_HW_\varepsilon\|_2\|v\|_2^{\frac{1}{2}}(\|v\|_2^{\frac{1}{2}}+\|\nabla_Hv\|_2^{\frac{1}{2}})\|\Delta_Hv\|_2^{\frac{1}{2}}\|\nabla_H\Delta_Hv\|_2^{\frac{1}{2}}\\
&+C\varepsilon^2\|\Delta_Hv\|_2\|\partial_zW_\varepsilon\|_2^{\frac{1}{2}}(\|\partial_zW_\varepsilon\|_2^{\frac{1}{2}}+\|\nabla_H\partial_zW_\varepsilon\|_2^{\frac{1}{2}})\\
&\qquad\qquad\times\Big[\|\partial_zv\|_2^{\frac{1}{2}}(\|\partial_zv\|_2^{\frac{1}{2}}+\|\nabla_H\partial_zv\|_2^{\frac{1}{2}})+\|v\|_2^{\frac{1}{2}}(\|v\|_2^{\frac{1}{2}}+\|\nabla_Hv\|_2^{\frac{1}{2}})\Big]\\
&+C\varepsilon^2\|\Delta_HW_\varepsilon\|_2\|\nabla_Hv\|_2\|\Delta_Hv\|_2\\
&+C\varepsilon^2(\|\nabla_H\partial_zV_\varepsilon\|_2+\|\nabla_HV_\varepsilon\|_2)\|\nabla_Hv\|_2\|\Delta_Hv\|_2\\
&+C\varepsilon^2\|\nabla_H\partial_zv\|_2\|\partial_zW_\varepsilon\|_2^{\frac{1}{2}}(\|\partial_zW_\varepsilon\|_2^{\frac{1}{2}}+\|\nabla_H\partial_zW_\varepsilon\|_2^{\frac{1}{2}})
\|\nabla_Hv\|_2^{\frac{1}{2}}\|\Delta_Hv\|_2^{\frac{1}{2}}\\
\leq&\frac{\varepsilon^2}{16}\|\nabla\nabla_HW_\varepsilon\|_2^2+\frac{1}{32}\|\nabla\nabla_HV_\varepsilon\|_2^2\\
&+C\|\nabla V_\varepsilon\|_2^2+C\varepsilon^2\|\nabla W_\varepsilon\|_2^2+C\varepsilon^2\|v\|_{H^2}^3(\|v\|_{H^2}+\|\nabla_Hv\|_{H^2}).
\end{align*}
By the H$\ddot{o}$lder inequality, the incompressibility condition, and integrating by parts, one can obtain
\begin{equation*}
\begin{aligned}
&\varepsilon^2\int_\Omega\Big(\partial_t w\Delta W_\varepsilon-\Delta_Hw\Delta W_\varepsilon-\varepsilon^{\alpha-2}\partial_z^2w\Delta W_\varepsilon\Big)d\Omega\\
=&\varepsilon^2\int_\Omega\partial_tw\Delta_HW_\varepsilon d\Omega-\varepsilon^2\int_\Omega\partial_t\partial_zw\partial_zW_\varepsilon d\Omega-\varepsilon^2\int_\Omega\Delta_Hw\Delta_HW_\varepsilon d\Omega\\
&+\varepsilon^2\int_\Omega\Delta_H\partial_zw\partial_zW_\varepsilon d\Omega-\varepsilon^\alpha\int_\Omega\partial_z^2w\Delta_HW_\varepsilon d\Omega+\varepsilon^\alpha\int_\Omega\partial_z^3w\partial_zW_\varepsilon d\Omega\\
\leq&\frac{\varepsilon^2}{16}\|\Delta_HW_\varepsilon\|_2^2+C\varepsilon^2(\|\partial_tv\|_{H^1}^2+\|\nabla_Hv\|_{H^2}^2)+C\varepsilon^2\|\partial_zW_\varepsilon\|_2^2.
\end{aligned}
\end{equation*}
Now, collecting the above estimates yield
\begin{equation}
\begin{aligned}
&\frac{1}{2}\frac{d}{dt}\varepsilon^2\|\nabla W_\varepsilon\|_2^2+\frac{11}{16}\Big(\varepsilon^2\|\nabla\nabla_HW_\varepsilon\|_2^2+\varepsilon^\alpha\|\partial_z\nabla W_\varepsilon\|_2^2\Big)+\int_\Omega \partial_zP_\varepsilon\Delta W_\varepsilon~d\Omega\\
\leq&C\|\nabla V_\varepsilon\|_2^2(\|\nabla V_\varepsilon\|_2^2+\|\nabla\nabla_HV_\varepsilon\|_2^2)+\frac{1}{8}\|\nabla\nabla_HV_\varepsilon\|_2^2\\
&+C\varepsilon^2\|\nabla W_\varepsilon\|_2^2(\varepsilon^2\|\nabla W_\varepsilon\|_2^2+\varepsilon^2\|\nabla\nabla_HW_\varepsilon\|_2^2)\\
&+C\|V_\varepsilon\|_2^2(\|V_\varepsilon\|_2^2+\|\nabla_HV_\varepsilon\|_2^2)+C(\|\nabla V_\varepsilon\|_2^2+\varepsilon^2\|\nabla W_\varepsilon\|_2^2)(\|v\|_{H^2}^4+1)\\
&+C(\varepsilon^2\| W_\varepsilon\|_2^2+\|V_\varepsilon\|_2^2)(\|v\|_{H^2}^4+\|v\|_{H^2}^2\|\nabla_Hv\|_{H^2}^2+1)\\
&+C\varepsilon^2\|v\|_{H^2}^3(\|v\|_{H^2}+\|\nabla_Hv\|_{H^2})
+C\varepsilon^2(\|\partial_tv\|_{H^1}^2+\|\nabla_Hv\|_{H^2}^2).\label{22}
\end{aligned}
\end{equation}
Combining \eqref{21}, \eqref{22} and by Proposition \ref{pro4},  one gets
\begin{equation*}
\begin{aligned}
&\frac{1}{2}\frac{d}{dt}\Big(\varepsilon^2\|\nabla W_\varepsilon\|_2^2+\|\nabla V_\varepsilon\|_2^2\Big)\\
&\qquad+\frac{5}{8}\Big(\varepsilon^2\|\nabla\nabla_HW_\varepsilon\|_2^2+\varepsilon^\alpha\|\partial_z\nabla W_\varepsilon\|_2^2+\|\nabla\nabla_HV_\varepsilon\|_2^2+\varepsilon^{\alpha-2}\|\nabla\partial_zV_\varepsilon\|_2^2\Big)\\
\leq&C_1(\|\nabla V_\varepsilon\|_2^2+\varepsilon^2\|\nabla W_\varepsilon\|_2^2)(\|\nabla V_\varepsilon\|_2^2+\|\nabla\nabla_HV_\varepsilon\|_2^2+\varepsilon^2\|\nabla W_\varepsilon\|_2^2+\varepsilon^2\|\nabla\nabla_HW_\varepsilon\|_2^2)\\
&+C(\varepsilon^2\|\nabla W_\varepsilon\|_2^2+\|\nabla V_\varepsilon\|_2^2)(G^2(t)+1)+C\|V_\varepsilon\|_2^2(\|V_\varepsilon\|_2^2+\|\nabla_HV_\varepsilon\|_2^2)\\
&+C(\|V_\varepsilon\|_2^2+\varepsilon^2\| W_\varepsilon\|_2^2)(G^2(t)+G(t)\|\nabla_Hv\|_{H^2}^2+1)\\
&+C\varepsilon^2G^{\frac{3}{2}}(t)(G^{\frac{1}{2}}(t)+\|\nabla_Hv\|_{H^2})
+C\varepsilon^{\alpha-2}G(t)+C\varepsilon^2(\|\partial_tv\|_{H^1}^2+\|\nabla_Hv\|_{H^2}^2),
\end{aligned}
\end{equation*}
from which, by the assumption $\sup_{0\leq s\leq t}\big(\|\nabla V_\varepsilon\|_2^2+\varepsilon^2\|\nabla W_\varepsilon\|_2^2\big)(s)\leq\sigma^2$, letting $\sigma^2=\frac{1}{16C_1}$, and recalling (\ref{***}), one can see
\begin{equation*}
\begin{aligned}
&\frac{d}{dt}\Big(\varepsilon^2\|\nabla W_\varepsilon\|_2^2+\|\nabla V_\varepsilon\|_2^2\Big)\\
&\qquad+\Big(\varepsilon^2\|\nabla\nabla_HW_\varepsilon\|_2^2+\varepsilon^\alpha\|\partial_z\nabla W_\varepsilon\|_2^2+\|\nabla\nabla_HV_\varepsilon\|_2^2+\varepsilon^{\alpha-2}\|\nabla\partial_zV_\varepsilon\|_2^2\Big)\\
\leq&C(\varepsilon^2\|\nabla W_\varepsilon\|_2^2+\|\nabla V_\varepsilon\|_2^2)(G^2(t)+K_1(t)\varepsilon^\beta+1)\\
&+CK_1(t)\varepsilon^\beta[K_1(t)\varepsilon^\beta+G^2(t)+G(t)\|\nabla_Hv\|_{H^2}^2+1]\\
&+C\varepsilon^\beta(G^3(t)+\|\partial_tv\|_{H^1}^2+\|\nabla_Hv\|_{H^2}^2+1).
\end{aligned}
\end{equation*}
Recalling $(V_\varepsilon,W_\varepsilon)|_{t=0}=0$, it follows from the Gronwall inequality and Proposition \ref{pro4} that
\begin{equation*}
\begin{aligned}
&\sup_{0\leq s\leq t}\big(\|\nabla V_\varepsilon\|_2^2+\varepsilon^2\|\nabla W_\varepsilon\|_2^2\big)(s)\\
&\qquad+\int_0^{t}\Big(\|\nabla\nabla_HV_\varepsilon\|_2^2+\varepsilon^2\|\nabla\nabla_HW_\varepsilon\|_2^2+\varepsilon^{\alpha-2}\|\nabla\partial_zV_\varepsilon\|_2^2+\varepsilon^\alpha\|\nabla\partial_zW_\varepsilon\|_2^2\Big)ds\\
\leq&C\varepsilon^\beta e^{Ct(G^2(t)+K_1(t)\varepsilon^\beta+1)}\Big[t\big(K_1^2(t)+G^4(t)+1\big)+K_1(t)G(t)+1\Big]
:=K_2(t)\varepsilon^\beta,
\end{aligned}
\end{equation*}
proving the conclusion.
\end{proof}

The next proposition shows that the smallness condition of $(\nabla_HV_\varepsilon,\varepsilon W_\varepsilon)$ in Proposition \ref{pro5.1} holds for any finite time $T>0$ provided $\varepsilon\in(0,\varepsilon_T)$, where $\varepsilon_T$ is a positive constant depending on $T$. As a result,  the local strong solution $(v_\varepsilon,w_\varepsilon)$ of \eqref{2} exists in $[0,T]$ for $\varepsilon\in(0,\varepsilon_T)$.
\begin{proposition}
\label{pro5.2}
Let $T_\varepsilon^*$ be the maximal existence time of the unique local strong solution $(v_\varepsilon,w_\varepsilon)$ to \eqref{2}, subject to \eqref{3}--\eqref{5}. Then for any finite time $T>0$, there exists a positive constant $\varepsilon_T$ depending only on $\|v_0\|_{H^2}$, $T$, $L_1$ and $L_2$, such that $T<T_\varepsilon^*$, as long as $\varepsilon\in (0,\varepsilon_T)$, and that
\begin{equation*}
\begin{aligned}
&\sup_{0\leq t\leq T}\big(\| V_\varepsilon\|_{H^1}^2+\varepsilon^2\| W_\varepsilon\|_{H^1}^2\big)(t)\\
&\qquad+\int_0^{T}\Big(\|\nabla_HV_\varepsilon\|_{H^1}^2+\varepsilon^2\|\nabla_HW_\varepsilon\|_{H^1}^2+\varepsilon^{\alpha-2}\|\partial_zV_\varepsilon\|_{H^1}^2+\varepsilon^\alpha\|\partial_zW_\varepsilon\|_{H^1}^2\Big)dt\\
&\leq K_3(T)\varepsilon^\beta,
\end{aligned}
\end{equation*}
where $K_3(t)$ is a nonnegative continuously increasing function on $[0,\infty)$ determined only by $\|v_0\|_{H^2}$, $L_1$ and $L_2$.
\end{proposition}

\begin{proof}
Set $T^{**}_\varepsilon=\min\{T,T^*_\varepsilon\}$. Then, by (\ref{***}), one has
\begin{equation}
\begin{aligned}
&\sup_{0\leq t\leq T^{**}_\varepsilon}\big(\|V_\varepsilon\|_2^2+\varepsilon^2\|W_\varepsilon\|_2^2\big)(t)\\
&\qquad+\int_0^{T^{**}_\varepsilon}\Big(\|\nabla_HV_\varepsilon\|_2^2+\varepsilon^2\|\nabla_HW_\varepsilon\|_2^2+\varepsilon^{\alpha-2}\|\partial_zV_\varepsilon\|_2^2+\varepsilon^\alpha\|\partial_zW_\varepsilon\|_2^2\Big)dt\\
&\leq K_1(T)\varepsilon^\beta.\label{23}
\end{aligned}
\end{equation}

Let $\sigma$ be the constant in Proposition \ref{pro5.1}. Define
\begin{equation*}
t_\varepsilon^*:=\sup\Big\{t\in(0,T^{**}_\varepsilon)\Big|\sup_{0\leq s\leq t}(\|\nabla V_\varepsilon\|_2^2+\varepsilon^2\|\nabla W_\varepsilon\|_2^2)\leq \sigma^2\Big\}.
\end{equation*}
By Proposition \ref{pro5.1}, one can obtain
\begin{equation}
\begin{aligned}
&\sup_{0\leq s\leq t}\big(\|\nabla V_\varepsilon\|_2^2+\varepsilon^2\|\nabla W_\varepsilon\|_2^2\big)(s)\\
&~~~+\int_0^{t}\Big(\|\nabla\nabla_HV_\varepsilon\|_2^2+\varepsilon^2\|\nabla\nabla_HW_\varepsilon\|_2^2+\varepsilon^{\alpha-2}\|\nabla\partial_zV_\varepsilon\|_2^2+\varepsilon^\alpha\|\nabla\partial_zW_\varepsilon\|_2^2\Big)ds\\
&\leq K_2(t)\varepsilon^\beta\leq K_2(T)\varepsilon^\beta\leq \frac{\sigma^2}{2}, \label{24}
\end{aligned}
\end{equation}
for any $t\in[0,t^*_\varepsilon)$ and for any $\varepsilon\in(0,\varepsilon_T)$, where $\varepsilon_T=\Big(\frac{\sigma^2}{2K_2(T)}\Big)^{\frac{1}{\beta}}$. Therefore,

\begin{equation}
\sup_{0\leq t< t_\varepsilon^*}\big(\|\nabla V_\varepsilon\|_2^2+\varepsilon^2\|\nabla W_\varepsilon\|_2^2\big)(t)\leq
\frac{\sigma^2}{2},\qquad\forall\varepsilon\in(0,\varepsilon_T).\label{25}
\end{equation}
By the definition of $t_\varepsilon^*$, this implies $t_\varepsilon^*=T^{**}_\varepsilon$ and, consequently, \eqref{24} holds for any $t\in[0,T_\varepsilon^{**})$.

We claim that $T^{**}_\varepsilon\geq T$ for any $\varepsilon\in(0,\varepsilon_T)$. Assume in contradiction that  $T_\varepsilon^{**}<T$, i.e., $T^*_\varepsilon<T$. This implies the maximal existence time of $(v_\varepsilon,w_\varepsilon)$ is finite and, consequently, recalling Proposition \ref{pro4}, it must have
\begin{equation*}
\limsup_{t\rightarrow (T_\varepsilon^{*})^-}(\|\nabla V_\varepsilon\|_2^2+\varepsilon^2\|\nabla W_\varepsilon\|_2^2)=\infty,
\end{equation*}
which contradicts to \eqref{24}. This contradiction implies $T^{**}_\varepsilon\geq T$ and thus $T^{*}_\varepsilon\geq T$. Thanks this and combining \eqref{23} and \eqref{24}, one obtains
\begin{equation*}
\begin{aligned}
&\sup_{0\leq t\leq T}\big(\| V_\varepsilon\|_{H^1}^2+\varepsilon^2\| W_\varepsilon\|_{H^1}^2\big)(t)\\
&~~~+\int_0^{T}\Big(\|\nabla_HV_\varepsilon\|_{H^1}^2+\varepsilon^2\|\nabla_HW_\varepsilon\|_{H^1}^2+\varepsilon^{\alpha-2}\|\partial_zV_\varepsilon\|_{H^1}^2+\varepsilon^\alpha\|\partial_zW_\varepsilon\|_{H^1}^2\Big)dt\\
&\leq (K_1(T)+K_2(T))\varepsilon^\beta:=K_3(T)\varepsilon^\beta.
\end{aligned}
\end{equation*}
This proves the conclusion.
\end{proof}

\begin{proof}[Proof of Theorem \ref{thm2}.]

For any finite time $T>0$, let $\varepsilon_T$ be the constant in Proposition \ref{pro5.2}. Then, by Proposition \ref{pro5.2}, for any $\varepsilon\in(0,\varepsilon_T)$, the scaled Navier-Stokes system (\ref{2})--(\ref{5}) exists a unique strong solution $(v_\varepsilon,w_\varepsilon)$ in $[0,T]$. While the following estimate holds
\begin{equation*}
\begin{aligned}
&\sup_{0\leq t\leq T}\big(\| V_\varepsilon\|_{H^1}^2+\varepsilon^2\| W_\varepsilon\|_{H^1}^2\big)(t)\\
&\qquad+\int_0^{T}\Big(\|\nabla_HV_\varepsilon\|_{H^1}^2+\varepsilon^2\|\nabla_HW_\varepsilon\|_{H^1}^2+\varepsilon^{\alpha-2}\|\partial_zV_\varepsilon\|_{H^1}^2+\varepsilon^\alpha\|\partial_zW_\varepsilon\|_{H^1}^2\Big)dt\\
&\leq K_3(T)\varepsilon^\beta.
\end{aligned}
\end{equation*}
which is exactly estimate stated in Theorem \ref{thm2}. The convergences are the direct corollaries of the above estimate. This completes the proof of Theorem \ref{thm2}.
\end{proof}

\smallskip
{\bf Acknowledgment.}
The work of J.L. was supported in part by the National
Natural Science Foundation of China (11971009, 11871005, and 11771156) and by
the Guangdong Basic and Applied Basic Research Foundation (2019A1515011621,
2020B1515310005, 2020B1515310002, and 2021A1515010247). The work of E.S.T. was
supported in part by the Einstein Stiftung/Foundation-Berlin, through the Einstein Visiting Fellow Program.

\bigskip

\end{document}